%% file: main.tex
\pdfoutput=1
\documentclass[12pt,a4paper]{amsart}
\usepackage{tgtermes}
\calclayout
\usepackage{amscd,amsmath,amsxtra,amsthm,amssymb,stmaryrd,xr,mathrsfs,mathtools,enumerate,commath,comment,tikz-cd,float,xcolor,stmaryrd}
\usepackage{longtable} % for 'longtable' environment
\usepackage{pdflscape} % for 'landscape' environment
\usepackage{booktabs}
\usepackage[backref=true,giveninits=true,doi=false,url=false,isbn=false,backend=biber,style=alphabetic,maxbibnames=100,sorting=nyt]{biblatex}
    \newbibmacro{string+doiurlisbn}[1]{%
      \iffieldundef{doi}{%
        \iffieldundef{url}{%
          \iffieldundef{isbn}{%
            \iffieldundef{issn}{%
              #1%
            }{%
              \href{http://books.google.com/books?vid=ISSN\thefield{issn}}{#1}%
            }%
          }{%
            \href{http://books.google.com/books?vid=ISBN\thefield{isbn}}{#1}%
          }%
        }{%
          \href{\thefield{url}}{#1}%
        }%
      }{%
        \href{http://dx.doi.org/\thefield{doi}}{#1}%
      }%
    }
    \DeclareFieldFormat{title}{\usebibmacro{string+doiurlisbn}{\mkbibemph{#1}}}
    \DeclareFieldFormat[article,incollection,thesis,misc,inproceedings,online,inbook]{title}{\usebibmacro{string+doiurlisbn}{\mkbibquote{#1}}}
    \DeclareFieldFormat{year}{\usebibmacro{year-parenthesis}{\mkbibemph{#1}}}
\renewbibmacro*{date}{%
\iffieldundef{year}
{}
{\ifentrytype{misc}{\printtext[parens]{\printdate}}{\printdate}}}%
\usepackage{hyperref}
\usepackage[textsize=tiny,textwidth=1in]{todonotes}
\definecolor{vegasgold}{rgb}{0.77, 0.7, 0.35}
\definecolor{darkgoldenrod}{rgb}{0.72, 0.53, 0.04}
\definecolor{gold(metallic)}{rgb}{0.83, 0.69, 0.22}
\hypersetup{
 colorlinks=true,
 linkcolor=darkgoldenrod,
 filecolor=brown,      
 urlcolor=darkgoldenrod,
 citecolor=darkgoldenrod,
 pdftitle={Spanning trees in Z-covers of a finite graph and Mahler measures}
 }
\usepackage[capitalize,nameinlink,noabbrev]{cleveref}
\usepackage{tikz}

\usetikzlibrary{decorations.markings}

\usetikzlibrary{shapes.geometric}
\tikzset{every loop/.style={min distance=10mm,looseness=10}}

\tikzcdset{scale cd/.style={every label/.append style={scale=#1},
    cells={nodes={scale=#1}}}}
\DeclareFontFamily{U}{wncy}{}
\DeclareFontShape{U}{wncy}{m}{n}{<->wncyr10}{}
\DeclareSymbolFont{mcy}{U}{wncy}{m}{n}
\DeclareMathSymbol{\Sh}{\mathord}{mcy}{"58}
\usepackage[T2A,T1]{fontenc}
\usepackage[OT2,T1]{fontenc}

\newtheorem{theorem}{Theorem}[section]

\newtheorem{proposition}[theorem]{Proposition}
\newtheorem{corollary}[theorem]{Corollary}

\newtheorem{assumption}[theorem]{Assumption}

\numberwithin{equation}{section}

\theoremstyle{remark}
\newtheorem{remark}[theorem]{Remark}
\newtheorem{example}[theorem]{Example}

\newcommand{\ord}{\mathrm{ord}}

\newcommand{\Z}{\mathbb{Z}}

\newcommand{\N}{\mathbb{N}}

\begin{document}
\title[Spanning trees in $\mathbb{Z}$-covers of a finite graph and Mahler measures]{Spanning trees in $\mathbb{Z}$-covers of a finite graph and Mahler measures}

\author[R.~Pengo]{Riccardo Pengo}
\address{
	Riccardo Pengo - Institut für Algebra, Zahlentheorie und Diskrete Mathematik, Fakultät für Mathematik und Physik, Leibniz Universität Hannover, Welfengarten 1, 30167 Hannover, Germany}
\email{\href{mailto:pengo@math.uni-hannover.de}{\textcolor{black}{pengo@math.uni-hannover.de}}}

\author[D.~Valli\`{e}res]{Daniel Valli\`{e}res}
\address{Daniel Valli\`{e}res - Mathematics and Statistics Department, California State University, Chico, CA 95929, USA}
\email{\href{mailto:dvallieres@csuchico.edu}{\textcolor{black}{dvallieres@csuchico.edu}}}

\begin{abstract}
Using the special value at $u=1$ of Artin-Ihara $L$-functions, we associate to every $\mathbb{Z}$-cover of a finite connected graph a polynomial which we call the \emph{Ihara polynomial}. 
We show that the number of spanning trees for the finite intermediate graphs of such a cover can be expressed in terms of the Pierce-Lehmer sequence associated to a factor of the Ihara polynomial.  This allows us to express the asymptotic growth of the number of spanning trees in terms of the Mahler measure of this polynomial.  Specializing to the situation where the base graph is a bouquet or the dumbbell graph gives us back previous results in the literature for circulant and $I$-graphs (including the generalized Petersen graphs).  We also express the $p$-adic valuation of the number of spanning trees of the finite intermediate graphs in terms of the $p$-adic Mahler measure of the Ihara polynomial.  When applied to a particular $\mathbb{Z}$-cover, our result gives us back Lengyel's calculation of the $p$-adic valuations of Fibonacci numbers.   
\end{abstract}

\date{\today} 

\maketitle
\tableofcontents 

%\section{Introduction}
\section{Introduction}
  
The aim of the present paper is to explain how the number of spanning trees in a $\mathbb{Z}$-cover of finite graphs evolves, by providing an explicit recipe to compute the invariants which describe this evolution in terms of a polynomial that can be associated to the cover in question.

\subsection{Historical remarks}
Before describing in detail the main results of this paper, let us provide some overview of the main questions which motivated the present paper. 

Iwasawa theory is concerned with the study of the evolution of certain invariants within a tower of objects (see \cite{Greenberg_2001} for a comprehensive survey).  The first example of this is provided by the evolution, as $n \to +\infty$, of the group of $\mathbb{F}_n$-rational points of the Jacobian of a curve defined over a finite field $\mathbb{F}$, where $\mathbb{F}_n \supseteq \mathbb{F}$ is the unique extension (up to isomorphism) of $\mathbb{F}$ having degree $n$. This example was studied by Weil and led him to formulate his celebrated conjectures concerning the properties of the zeta functions associated to varieties defined over a finite field.  Iwasawa pursued analogous investigations concerning the evolution of class groups of number fields in a tower of cyclotomic extensions, which are akin to the extensions of a function field that are obtained by increasing the field of constants, as explained in \cite[Page~188]{Rosen:2002}.  This initiated a large series of works, which study the evolution of different invariants along a tower of number fields whose Galois group is a $p$-adic Lie group (see \cite{Kakde_2011} for a survey).
Moreover, Iwasawa theory has been extended to the study of the evolution of invariants of many different arithmetic objects, such as elliptic curves or even general motives (see \cite{Fukaya_Kato_2006} for one of the most general frameworks available at present).

In a somehow different direction, ideas from Iwasawa theory have found applications also outside number theory and algebraic geometry. 
More precisely, it has been shown that the torsion subgroups of the first homology groups of a tower of hyperbolic $3$-manifolds whose base is the complement of a knot or a link evolve according to a pattern which is very similar to the one appearing in Iwasawa theory (see \cite{Hillman_Matei_Morishita_2006,Kadokami_Mizusawa_2008,Ueki_2017}).
Considering hyperbolic manifolds allows one to study towers whose groups of deck transformations are not necessarily profinite, which is not possible when one studies towers of number fields. 
For instance, one can consider a $\mathbb{Z}$-cover of hyperbolic manifolds.
In this case, when the base of the tower consists of the complement of a knot in the three dimensional sphere, the Alexander polynomial of the knot in question can be used to describe explicitly the growth of the torsion inside the first homology groups of the manifolds in question, as proven by Ueki \cite{Ueki_2020} in the $p$-adic case, and by González-Acuña and Short \cite{Gonzalez-Acuna_Short_1991} and Riley \cite{Riley_1990} in the Archimedean case.
These results are particularly interesting in view of the widely explored analogy between number fields and knots (see \cite{Morishita_2012} for a survey).

Finally, an analogue of Iwasawa theory has also been developed to study the evolution of the so-called Picard group of degree zero of a finite connected graph $X$, as one moves along a tower. 
This finite group, defined for instance in \cite[Section~1.3]{Corry_Perkinson_2018},
is analogous to the class group of a number field, or to the Picard group of degree zero of a curve defined over a finite field. 
Its cardinality, usually denoted by $\kappa(X)$, is given by the number of spanning trees of the graph in question.
The evolution of this number when the finite graph in question varies along a tower has been the subject of a series of papers written by several authors in collaboration with the second named author of the present paper \cite{Vallieres:2021,McGown/Vallieres:2022,McGown/Vallieres:2022a,Lei/Vallieres:2023,DuBose/Vallieres:2022}.
More precisely, if $\ell \in \mathbb{N}$ is a rational prime and
\begin{equation} \label{eq:tower}
    \dots \to X_{\ell^n} \to \dots \to X_{\ell} \to X_1 = X
\end{equation}
is a tower of finite graphs, such that each $X_{\ell^n}/X$ is a Galois cover with Galois group $\mathbb{Z}/\ell^n \mathbb{Z}$, 
it was shown in \cite{Vallieres:2021,McGown/Vallieres:2022, McGown/Vallieres:2022a} that there exist non-negative integers $\mu_{\ell},\lambda_{\ell}$ and an integer $\nu_{\ell}$ such that
\begin{equation} \label{eq:Iwasawa}
{\rm ord}_{\ell}(\kappa(X_{\ell^n})) = \mu_{\ell} \cdot \ell^{n} + \lambda_{\ell} \cdot n + \nu_{\ell}, 
\end{equation}
for $n$ large enough, where $\ord_\ell$ denotes the usual $\ell$-adic valuation on $\mathbb{Q}$.
Furthermore, it was shown in \cite{Lei/Vallieres:2023} that if $p$ is another rational prime different than $\ell$, then there exist a non-negative integer $\mu_{p}$ and an integer $\nu_{p}$ such that
\begin{equation} \label{eq:Washington}
{\rm ord}_{p}(\kappa(X_{\ell^n})) = \mu_{p} \cdot \ell^{n} + \nu_{p}, 
\end{equation}
for $n$ large enough.
On the other hand, given an integer $d \geq 1$, and a tower of finite graphs
\[
    \dots \to X_{\ell^n}^{(d)} \to \dots \to X_\ell^{(d)} \to X_1 = X,
\]
such that each $X_{\ell^n}^{(d)}/X$ is a Galois cover with Galois group $(\mathbb{Z}/\ell^n \mathbb{Z})^d$, it was shown in \cite{DuBose/Vallieres:2022} that there exists a polynomial $P \in \mathbb{Q}[t_1,t_2]$ of total degree at most $d$, and linear in $t_2$, such that
\begin{equation} \label{eq:Greenberg}
    \ord_\ell(\kappa(X_{\ell^n}^{(d)})) = P(\ell^n,n),
\end{equation}
for every $n$ which is large enough.
These advances in the Iwasawa theory of finite graphs can be seen as being analogous to more classical theorems and conjectures in the Iwasawa theory of number fields.
More precisely, \eqref{eq:Iwasawa} is analogous to a classical theorem of Iwasawa \cite{Iwasawa:1959} concerning $\mathbb{Z}_\ell$-extensions of number fields, whereas \eqref{eq:Washington} is analogous to a result of Washington for the cyclotomic $\mathbb{Z}_\ell$-extension of an abelian number field, proved in \cite{Washington:1978}, and \eqref{eq:Greenberg} is akin to a conjecture of Greenberg, which is discussed by Cuoco and Monsky in \cite[Section~7]{Cuoco/Monsky:1981}.

The results (\ref{eq:Iwasawa}) and (\ref{eq:Greenberg}) were originally proven by working on the ``analytic side'' of Iwasawa theory, \emph{i.e.} by constructing appropriate elements of the Iwasawa algebra 
\[\mathbb{Z}_\ell\llbracket \mathbb{Z}_\ell^d \rrbracket \cong \mathbb{Z}_\ell\llbracket T_1,\dots,T_d \rrbracket.\] 
On the other hand, Gonet \cite{Gonet:2021a,Gonet:2022} reproved \eqref{eq:Iwasawa} using a module theoretical approach, which was recently shown to be closely related to the analytic approach in the work of Kleine and Müller \cite{Kleine/Muller:2022}. More precisely, this work proves an analogue of the Iwasawa main conjecture in the setting of graphs, which allows Kleine and Müller to prove \eqref{eq:Greenberg} in an algebraic way. 
Moreover, Kataoka's recent work \cite{Kataoka:2023} studies the Fitting ideals that appear in this setting, and Kleine and Müller's more recent work \cite{Kleine/Muller:2023} shows how to adapt some of these ideas to the non abelian setting.

To conclude this overview, let us mention that the recent work of Lei and Müller \cite{Lei_Mueller_2023a,Lei_Mueller_2023b} shows how one can obtain natural towers of finite graphs by looking at the isogeny graphs associated to elliptic curves defined over a finite field $\mathbb{F}$. More precisely, in \cite{Lei_Mueller_2023a} the authors consider $\ell$-isogeny graphs $\tilde{G}_N^m$ of ordinary elliptic curves with a $\Gamma_1(N p^m)$-level structure, where $p$ is the characteristic of $\mathbb{F}$, while $\ell$ is a prime different from $p$ and $N$ is a fixed integer coprime to $p$. In particular, they fix an ordinary elliptic curve $E$ defined over $\mathbb{F}$, which also admits a non-trivial $\ell$-isogeny defined over $\mathbb{F}$, and they prove that there exists an integer $m_0$ such that the connected components $(\tilde{\mathcal{G}}_N^m)_{m = m_0}^{+\infty}$ of the graphs $(\tilde{G}_N^m)_{m = m_0}^{+\infty}$ which contain a vertex corresponding to $E$ give rise to a $\mathbb{Z}_p$-tower.
These graphs generalize the celebrated isogeny volcanoes, which are vastly used in cryptography, and have been classified in recent work of Bambury, Campagna and Pazuki \cite{Bambury_Campagna_Pazuki_2022}.
In the subsequent paper \cite{Lei_Mueller_2023b}, Lei and Müller considered $\ell$-isogeny graphs of elliptic curves with full $\Gamma(N p^n)$-level structures, and they showed that their ordinary connected components do not give rise to Galois covers, while their supersingular ones do, at least when $N \leq 2$ and for a positive proportion of primes $p$.
In this case, the resulting tower has a non-abelian Galois group, isomorphic to $\mathrm{GL}_2(\mathbb{Z}_p)$, which therefore fits into the framework developed by Kleine and Müller in \cite{Kleine/Muller:2023}.

\subsection{Main results}
Inspired by the results mentioned in the previous section, we show in the present paper how the invariants appearing in \eqref{eq:Iwasawa}
and \eqref{eq:Washington} can be explicitly computed when \eqref{eq:tower} comes from a $\mathbb{Z}$-cover of finite graphs.
More precisely, every Galois cover of graphs $Y/X$ with Galois group $G$ can be constructed from a voltage assignment, which is a function $\alpha \colon \mathbf{E}_X \to G$ such that $\alpha(\overline{e}) = \alpha(e)^{-1}$ for every $e \in \mathbf{E}_X$, where $\mathbf{E}_X$ denotes the set of directed edges of $X$, and $\overline{e}$ denotes the inverse of an edge (see \cref{loc_fin_graph} for further details).
Indeed, if $G$ is an arbitrary group and $\alpha \colon \mathbf{E}_X \to G$ is a voltage assignment, one can construct a graph $X(G,\alpha)$, introduced by Gross in \cite{Gross:1974}, which generalizes the usual notion of a Cayley graph, and is endowed with a canonical map $X(G,\alpha) \to X$, which is a Galois cover if and only if $X(G,\alpha)$ is connected, in which case $\mathrm{Gal}(X(G,\alpha)/X) \cong G$.
On the other hand, if $Y/X$ is a Galois cover of finite graphs, with Galois group $G$, there exists a voltage assignment $\alpha \colon \mathbf{E}_X \to G$ and an isomorphism of covers $Y/X \cong X(G,\alpha)/X$, as outlined in \cite[Section~3]{DuBose/Vallieres:2022}.

Now, let $G$ be an arbitrary group, and $\alpha \colon \mathbf{E}_X \to G$ be a voltage assignment. Then, for every normal subgroup $H \trianglelefteq G$ which has finite index, one has a finite graph $X_H := X(G/H,\alpha_H)$, where $\alpha_H \colon \mathbf{E}_X \to G/H$ denotes the voltage assignment obtained by composing $\alpha$ with the natural projection map $\pi \colon G \twoheadrightarrow G/H$.
If each of the finite graphs $X_H$ is connected, then it is a Galois cover of $X$, whose Galois group is canonically isomorphic to $G/H$.
In this setting, one of the main goals, which is related to the results mentioned above, is to describe how the number of spanning trees $\kappa(X_H)$ depends on $H$.
When $G = \mathbb{Z}_\ell^d$ for some $d \geq 1$, this is the content of the results which we recalled in the previous section, that lead to \eqref{eq:Iwasawa}, \eqref{eq:Washington} and \eqref{eq:Greenberg}.

As we mentioned above, in this paper we focus on the case $G = \mathbb{Z}$, and we provide a global analogue of the results obtained in \eqref{eq:Iwasawa} and \eqref{eq:Washington}.
More precisely, for every finite graph $X$ and every voltage assignment $\alpha \colon \mathbf{E}_X \to \mathbb{Z}$ such that each finite graph $X_n := X(\mathbb{Z}/n \mathbb{Z},\alpha_n)$ is connected, where $\alpha_n := \alpha_{n \mathbb{Z}}$, we show in \cref{thm:kN_Pierce_Lehmer} that the number of spanning trees $\kappa(X_n)$ of the graph $X_n$ is intimately related to the Pierce-Lehmer sequence $\{ \Delta_n(J_\alpha) \}_{n=1}^{+\infty}$ associated to a factor $J_\alpha \in \mathbb{Z}[t]$ of the \emph{Ihara polynomial} $\mathcal{I}_\alpha \in \mathbb{Z}[t^{\pm 1}]$, which is a Laurent polynomial that can be explicitly constructed from the voltage assignment $\alpha$, as we explain in \cref{growth}.
The Archimedean and $p$-adic absolute values of the aforementioned Pierce-Lehmer sequence, introduced by Pierce \cite{Pierce_1916} and Lehmer \cite{Lehmer_1933}, turn out to be related to the Archimedean and $p$-adic Mahler measures of the polynomial $\mathcal{I}_\alpha$, as we explain in \cref{sec:Pierce_Lehmer}.
In particular, these Mahler measures provide the main term which explains the order of growth of the different absolute values of the Pierce-Lehmer sequence $\{ \Delta_n(J_\alpha) \}_{n=1}^{+\infty}$.
This suffices to describe the asymptotic behaviour of the Archimedean (respectively $p$-adic) absolute value of the number of spanning trees $\kappa(X_n)$ whenever no root of $\mathcal{I}_\alpha$ lies on the unit circle of $\mathbb{C}$ (respectively $\mathbb{C}_p$), as we explain in \cref{cor:archimedean_asypmtotics,cor:p-adic_asymptotics}. In particular, we show in \cref{ex:Mednykh_1,ex:Mednykh_2} that our Archimedean result generalizes previous work of A.D. Mednykh and I.A. Mednykh \cite{Mednykh:2018,Mednykh:2019}.

One may of course wonder about the behaviour of the Archimedean or $p$-adic absolute value of $\kappa(X_n)$ when $\mathcal{I}_\alpha$ has some of its roots on the Archimedean (or $p$-adic) unit circle.
In fact, this question is central to understand the behaviour of the sequence $\kappa(X_n)$, as for almost every prime $p$ all the roots of $\mathcal{I}_\alpha$ will lie on the $p$-adic unit circle. 
In the case of the Archimedean absolute value, one can only get some upper and lower bounds for $\kappa(X_n)$, but not an exact asymptotic, as follows from Weyl's equidistribution theorem (see \cref{rmk:Weyl}).
In the $p$-adic case, to understand the absolute value of $\kappa(X_n)$ one needs to take into account a correcting factor which is described in \cref{thm:PL_unit_circle}.  Doing so, we arrive at the following result which we now present in a simplified version.  For the precise formulation, see \cref{p_adic_val}.
\begin{theorem} \label{thm:main_thm_intro}
Let $X$ be a finite connected graph whose Euler characteristic $\chi(X)$ does not vanish, and $\alpha \colon \mathbf{E}_X \to \mathbb{Z}$ be a voltage assignment such that for every $n \geq 1$ the finite graph \[X_n := X(\mathbb{Z}/n \mathbb{Z},\alpha_n)\] is connected (which can be checked using \cref{connectedness}).  Let $\mathcal{I}_\alpha \in \mathbb{Z}[t^{\pm 1}]$ be the Ihara polynomial associated to $\alpha$, and set \[J_\alpha(t) := t^b (t-1)^{-e} \mathcal{I}_\alpha(t),\] where $b := -\ord_{t=0}(\mathcal{I}_\alpha)$, and $e := \ord_{t=1}(\mathcal{I}_\alpha)$.  Fix a rational prime $p \in \mathbb{N}$, and let
$$\mu_p(X,\alpha) = -m_p(J_\alpha)/\log(p),$$
where $m_p(J_\alpha)$ denotes the logarithmic $p$-adic Mahler measure of $J_\alpha$, defined as in \eqref{eq:log_padic_m}.  
Then, there exist two explicit functions 
\[
    \begin{aligned}
        \N &\to \Z_{\geq 0} \\ n &\mapsto \lambda_{p,n}(X,\alpha)
    \end{aligned} \quad \text{and} \quad
    \begin{aligned}
        \N &\to \Z \\
        n &\mapsto \nu_{p,n}(X,\alpha)
    \end{aligned}
\]
whose images are finite, and an integer $c_{p}(X,\alpha)$,
such that 
\begin{equation} \label{global_for}
{\rm ord}_{p}(\kappa(X_{n})) = \mu_{p}(X,\alpha) \cdot n + \lambda_{p,n}(X,\alpha) \cdot {\rm ord}_{p}(n) + \nu_{p,n}(X,\alpha) + c_{p}(X,\alpha)
\end{equation}
for all $n \in \mathbb{N}$.
\end{theorem}
Specializing \eqref{global_for} to the subsequence $\{\kappa(X_{p^{k}}) \}_{k=1}^{\infty}$ gives
$${\rm ord}_{p}(\kappa(X_{p^{k}})) = \mu_{p}(X,\alpha) \cdot p^{k} + \lambda_{p,p^{k}}(X,\alpha) \cdot k + \nu_{p,p^{k}}(X,\alpha) + c_{p}(X,\alpha) $$
and specializing to the subsequence $\{\kappa(X_{\ell^{k}}) \}_{k=1}^{\infty}$, where $\ell$ is another rational prime different from $p$, gives
$${\rm ord}_{p}(\kappa(X_{\ell^{k}})) = \mu_{p}(X,\alpha) \cdot \ell^{k} + \nu_{p,\ell^{k}}(X,\alpha) + c_{p}(X,\alpha). $$
After studying the dependency on $k$ of the constants $\lambda_{p,p^{k}}(X,\alpha), \nu_{p,p^{k}}(X,\alpha)$, and $\nu_{p,\ell^{k}}(X,\alpha)$, one gets back the formulas \eqref{eq:Iwasawa} and \eqref{eq:Washington}, as we explain in \cref{rmk:p_l_adic_graphs}.
Moreover, we obtain similar results by specializing \cref{thm:main_thm_intro} to sequences of integers divisible only by a finite number of primes, as we explain in \cref{cor:PL_Friedman}. These identities can be seen as analogous to a result proven by Friedman \cite{Friedman:1982} for cyclotomic extensions of number fields which are abelian over $\mathbb{Q}$. 

To conclude, we show in \cref{sec:fibonacci} that \cref{thm:main_thm_intro} allows one to recover a well known formula which computes the $p$-adic valuations of Fibonacci numbers, which is due to Lengyel \cite{Lengyel:1995}.

\subsection{Notations and conventions} \label{notations}
Let $p$ be a rational prime.  We let $\mathbb{C}_{p}$ denote a fixed completion of an algebraic closure of the $p$-adic rational numbers $\mathbb{Q}_{p}$.  As usual, $\lvert \cdot \rvert_{p}$ and ${\rm ord}_{p}$ denote the $p$-adic absolute value and the $p$-adic valuation on $\mathbb{C}_{p}$, respectively.  They are related via
$${\rm ord}_{p}(x) = -\frac{\log |x|_{p}}{\log p},  $$
for all $x \in \mathbb{C}_{p}$, and they are normalized so that ${\rm ord}_{p}(p) = 1$. We also denote by $\mathbb{C}$ the field of complex numbers, endowed with the usual Archimedean absolute value $\lvert \cdot \rvert_\infty$.

If $G$ is an abelian group, not necessarily finite, we let $G^{\vee} = {\rm Hom}_{\mathbb{Z}}(G,W_{\infty})$, where $W_{\infty}$ denotes the group of roots of unity in an algebraic closure $\overline{\mathbb{Q}} \subseteq \mathbb{C}$ of $\mathbb{Q}$.  An element of $G^{\vee}$ will be called a character of finite order.  Here, we depart from the usual notation, since $G^{\vee}$ is not necessarily the Pontryagin dual of $G$.  For each rational prime $p$, we fix once and for all an embedding $\overline{\mathbb{Q}} \hookrightarrow \mathbb{C}_{p}$.  Via these embeddings, we view the characters in $G^{\vee}$ as taking values in $\mathbb{C}_{p}$ once a rational prime $p$ has been fixed.  If $n$ is a positive integer, then we let $W_{n}$ denote the group of $n$th roots of unity.  The symbol $\mathbb{N} = \{1,2,\dots\}$ refers to the collection of all positive integers.

%\section{Pierce-Lehmer sequences}
\section{Mahler measures and Pierce-Lehmer sequences} \label{PL}
In this section, we remind the reader about the resultant of two polynomials, which appears in \cref{sec:resultant}, and about the $p$-adic and Archimedean Mahler measures of polynomials, which we treat in \cref{sec:mahler}.
Moreover, we devote \cref{sec:Pierce_Lehmer} to collect some results about Pierce-Lehmer sequences. 
In particular, \cref{thm:PL_generic,thm:PL_unit_circle} provide explicit formulas to compute the $p$-adic valuations of Pierce-Lehmer sequences.

\subsection{Resultant} \label{sec:resultant}
Let $F$ be a field and let 
$$p(t) = a_{m}t^{m} + \ldots + a_{0} = a_{m} \prod_{i=1}^{m}(t - \alpha_{i}) $$
and
$$q(t) = b_{n}t^{n} + \ldots + b_{0} = b_{n} \prod_{j=1}^{n}(t - \beta_{i})$$
be two polynomials in $F[t]$ of degree $m$ and $n$, respectively.  Here, the roots $\alpha_{i}$ and $\beta_{i}$ are assumed to be in a fixed algebraic closure of $F$.  The resultant ${\rm Res}(p,q)$ of $p$ and $q$ is defined to be
\begin{equation} \label{def_res}
{\rm Res}(p,q) = a_{m}^{n}b_{n}^{m} \prod_{i=1}^{m}\prod_{j=1}^{n}(\alpha_{i} - \beta_{j})
\end{equation}
and is easily seen to be an element of $F$.
Let $r(t)$ be another polynomial with coefficients in $F$.  From the definition (\ref{def_res}), the two properties 
$${\rm Res}(p,q) = (-1)^{mn} {\rm Res}(q,p) \text{ and } {\rm Res}(p\cdot r,q) = {\rm Res}(p,q) \cdot {\rm Res}(r,q)$$
follow immediately.  Furthermore, one has
$$a_{m}^{n}\prod_{i=1}^{m}q(\alpha_{i})={\rm Res}(p,q)= (-1)^{mn}b_{n}^{m}\prod_{j=1}^{n}p(\beta_{j})$$
which can be seen as an instance of Weil's reciprocity law for the projective line over $F$.
Finally, the resultant can also be defined as the determinant of the Sylvester matrix of $p$ and $q$, as shown for instance in \cite[Lemma 3.3.4]{Cohen:1993}. This allows one to define the resultant $\mathrm{Res}(f,g) \in R$ of any pair of polynomials $f, g \in R[t]$ which have coefficients in an arbitrary commutative ring with unity $R$.

\subsection{Mahler measure} \label{sec:mahler}
Recall that if
$$f(t) = a_{d}t^{d} + \ldots + a_{0} \in \mathbb{C}[t] $$
is a nonzero polynomial of degree $d$, which can be factorised as
$$f(t) = a_{d} \prod_{i=1}^{d}(t - \alpha_{i}) $$
for some $\alpha_1,\dots,\alpha_d \in \mathbb{C}$, then one defines its Archimedean Mahler measure to be
\begin{equation} \label{eq:Archimedean_Mahler_measure}
    M_{\infty}(f) := |a_{d}|_\infty \prod_{i=1}^{d} {\rm max}\{1,|\alpha_{i}|_\infty \} \in \mathbb{R}_{>0}.
\end{equation}
This invariant, originally studied by Lehmer \cite{Lehmer_1933}, was generalized by Mahler \cite{Mahler_1962} to polynomials with any number of variables.

Now, let $p$ be a rational prime and let
$$g(t) = b_{d}t^{d} + \ldots + b_{0} \in \mathbb{C}_{p}[t] $$
be a nonzero polynomial of degree $d$, which factors as
$$g(t) = b_{d} \prod_{i=1}^{d}(t - \beta_{i}) $$
for some $\beta_1,\dots,\beta_d \in \mathbb{C}_{p}$. 
Following \cite{Ueki_2020}, we define similarly the $p$-adic Mahler measure of $g(t)$ to be
$$M_{p}(g) := \lvert b_{d} \rvert_{p} \prod_{i=1}^{d} {\rm max}\{1,\lvert \beta_{i} \rvert_{p} \} \in \mathbb{R}_{>0}. $$
This invariant, and its Archimedean analogue, are clearly multiplicative.
Furthermore, the $p$-adic Mahler measure of a polynomial \[g(t) = \sum_{i = 0}^d b_i t^i \in \mathbb{C}_p[t]\] can be easily computed from its coefficients, thanks to the formula
\begin{equation} \label{eq:Ueki_p_adic_Mahler}
    M_{p}(g) = {\rm max}\{\lvert b_{i} \rvert_{p} \, | \, i = 0,\ldots,d \},
\end{equation}
which was proved by Ueki in \cite[Proposition~2.7]{Ueki_2020}.
Finally, we introduce the logarithmic Archimedean Mahler measure
\[
m_\infty(f) := \log(M_\infty(f))
\]
of a polynomial $f(t) \in \mathbb{C}[t]$, and analogously the logarithmic $p$-adic Mahler measure
\begin{equation} \label{eq:log_padic_m}
m_{p}(g) := \log(M_{p}(g))
\end{equation}
of a polynomial $g(t) \in \mathbb{C}_p[t]$.

\begin{remark}
We note in passing that the logarithmic $p$-adic Mahler measure introduced in \eqref{eq:log_padic_m} does not coincide with the $p$-adic logarithmic Mahler measure introduced by Besser and Deninger in \cite{Besser_Deninger_1999}, which is a $p$-adic number.
\end{remark}

%\subsection{Pierce-Lehmer sequences}
\subsection{Pierce-Lehmer sequences}
\label{sec:Pierce_Lehmer}

Let 
$$f(t) = a_{d}t^{d} + a_{d-1}t^{d-1} + \ldots + a_{0} \in \mathbb{Z}[t], $$
with $a_{d} \neq 0$ and write 
$$f(t) = a_{d}\prod_{i=1}^{d}(t - \alpha_{i}) $$
for some $\alpha_1,\dots,\alpha_d \in \overline{\mathbb{Q}}$.  
The associated Pierce-Lehmer sequence is defined to be
\begin{equation} \label{eq:Pierce_Lehmer}
%\begin{aligned}
\Delta_{n}(f) = a_{d}^{n}\prod_{i=1}^{d}(\alpha_{i}^{n} - 1) = {\rm Res}(f(t),t^{n}-1) \in \mathbb{Z}. 
%\end{aligned}
\end{equation}
Fix now a rational prime $p$ and an embedding $\overline{\mathbb{Q}} \hookrightarrow \mathbb{C}_{p}$, as we did in \cref{notations}, and view all the algebraic numbers $\alpha_1,\dots,\alpha_d$ as lying in $\mathbb{C}_{p}$ via this embedding.  
\begin{theorem} \label{thm:PL_generic}
With the notation as above, one has
$$|\Delta_{n}(f)|_{p} = M_{p}(f)^{n} \prod_{\substack{i=1 \\ |\alpha_{i}|_{p} = 1}}^{d} |\alpha_{i}^{n}-1|_{p}. $$
\end{theorem}
\begin{proof}
Noting that for $\alpha \in \mathbb{C}_{p}$ and $n \in \mathbb{N}$, one has
\begin{equation*}
|\alpha^{n} - 1|_{p} = 
\begin{cases}
|\alpha|_{p}^{n}, &\text{ if }  |\alpha|_{p}> 1;\\
1, & \text{ if } |\alpha|_{p} < 1,
\end{cases}
\end{equation*}
one calculates
\begin{equation*}
\begin{aligned}
|\Delta_{n}(f)|_{p} &= |a_{d}|_{p}^{n} \prod_{\substack{i=1 \\ |\alpha_{i}|_{p} > 1}}^{d}|\alpha_{i}|_{p}^{n} \prod_{\substack{i=1 \\ |\alpha_{i}|_{p} = 1}}^{d}|\alpha_{i}^{n} - 1|_{p}\\
&= M_{p}(f)^{n} \prod_{\substack{i=1 \\ |\alpha_{i}|_{p} = 1}}^{d}|\alpha_{i}^{n} - 1|_{p}.
\end{aligned}
\end{equation*}
\end{proof}
It follows that in order to determine the $p$-adic valuation of the numbers $\Delta_{n}(f)$, one needs to understand the $p$-adic valuation of numbers of the form $\alpha^n - 1$, where $n \in \mathbb{N}$ and $\alpha \in \overline{\mathbb{Q}}_{p} \subseteq \mathbb{C}_{p}$ is a $p$-adic number such that $\lvert \alpha \rvert_p = 1$.
The following theorem, which is inspired by \cite[Lemma~2.11]{Ueki_2020}, provides a first step in this direction.

\begin{theorem} \label{thm:PL_unit_circle}
Let $\alpha \in \overline{\mathbb{Q}}_{p}$ be such that $\lvert \alpha \rvert_p = 1$, and assume that $\alpha$ is not a root of unity.  Let $\mathfrak{m}$ be the maximal ideal of the valuation ring $\mathcal{O}$ of $\overline{\mathbb{Q}}_{p}$ and let $N(\alpha)$ be the multiplicative order of $\alpha$ modulo $\mathfrak{m}$.  Then, there exists a function $c:\mathbb{N} \rightarrow \mathbb{Q}$ such that $c(m)$ is constant for $m$ large and for which
\begin{equation} \label{stat1}
{\rm ord}_{p}(\alpha^{n}-1) =
\begin{cases}
0, &\text{ if } N(\alpha) \nmid  n; \\
{\rm ord}_{p}(n) + c({\rm ord}_{p}(n)), &\text{ if } N(\alpha) \,\, |\, \, n.
\end{cases}
\end{equation}
\end{theorem}
\begin{proof}
First of all, let us write
\begin{equation} \label{eq:power_root_of_unity}
    \lvert \alpha^n - 1 \rvert_p = \prod_{\zeta \in W_n} \lvert \alpha - \zeta \rvert_p
\end{equation}
where $W_n \subseteq \mathcal{O}^\times$ denotes the group of roots of unity of order dividing $n$.
Now, the natural embedding of the ring of Witt vectors of $\overline{\mathbb{F}}_p$ inside $\mathcal{O}$ gives rise to the Teichm\"{u}ller lift
\[\tau: \overline{\mathbb{F}}_{p}^{\times} \hookrightarrow \mathcal{O}^{\times},\]
which is a morphism of groups that sends any $\beta \in \overline{\mathbb{F}}_p^\times$ to a root of unity whose order coincides with the multiplicative order of $\beta$ in $\overline{\mathbb{F}}_p^\times$.
Therefore, if \[\pi:\mathcal{O} \twoheadrightarrow \mathcal{O}/\mathfrak{m} = \overline{\mathbb{F}}_{p}\] is the natural projection map, the root of unity $\xi = \tau( \pi(\alpha)) \in \mathcal{O}^\times$ has order $N = N(\alpha)$, and we have that
$\lvert \alpha - \xi \rvert_{p} < 1 $ because $\tau$ is a section of $\pi$.

We can then use this root of unity $\xi$ to write the following formula
\begin{equation} \label{key}
|\alpha^{n} - 1|_{p}
= \prod_{\zeta \in W_{n}} |\alpha - \xi + \xi - \zeta|_{p},
\end{equation}
which follows from \eqref{eq:power_root_of_unity}.
Using this formula, we can prove immediately the first part of \eqref{stat1}. Indeed, if $N \nmid n$ then for every $\zeta \in W_n$ the order of the root of unity $\zeta/\xi$ is not a power of $p$, because otherwise there would exist some $r \in \mathbb{N}$ such that $\zeta^{p^{r}} = \xi^{p^{r}}$, from which it would follow that $\xi^{np^{r}} = 1$, and thus that $N \, | \, np^{r}$. Since we know that $(N,p) = 1$, this would imply that $N \, | \, n$, contradicting our assumption. 
Therefore, \cite[Lemma~2.9]{Ueki_2020} implies that
$$|\xi - \zeta|_{p} = |1 - \zeta/\xi|_{p} = 1 $$
for all $\zeta \in W_{n}$, which entails that $\lvert \alpha - \xi + \xi - \zeta \rvert_p = 1$ for every $\zeta \in W_n$, because $\lvert \alpha - \xi \rvert_p < 1$ by construction. Finally, we see thanks to \eqref{key} that $|\alpha^{n} - 1|_{p} = 1$, which proves the first part of the statement (\ref{stat1}).

Suppose now that $N \, | \, n$.
As before, we have that $\lvert \xi - \zeta \rvert_p = 1$ unless the order of $\mu := \zeta/\xi$ is a power of $p$.
Therefore
\begin{equation} \label{eq:alphan_only_padic}
    \lvert \alpha^n - 1 \rvert_p = \prod_{\mu \in W_{p^m}} \lvert \alpha - \xi \mu \rvert_p = \prod_{\mu \in W_{p^m}} \lvert \alpha - \xi + \xi(1-\mu) \rvert_p,
\end{equation}
where $m := \ord_p(n)$.
Moreover, if $\mu$ has order $p^k$, for some $k \in \mathbb{N} \cup \{0\}$ such that
\[
    p^{k-1} (p-1) \ord_p(\alpha-\xi) > 1
\]
we have that $\lvert \xi (1-\mu) \rvert_p = \lvert 1 - \mu \rvert_p > \lvert \alpha-\xi \rvert_p$, as follows from the classical fact that 
\[\ord_p(1-\mu) = \frac{1}{p^{k-1} (p-1)}\] 
for every root of unity $\mu \in W_{p^m} \setminus \{1\}$ of order $p^k$, whose proof can be found for example in \cite[Lemma~2.9]{Ueki_2020}. 

Hence, if we define $s \in \mathbb{N} \cup \{0\}$ to be the minimal non-negative integer such that
\begin{equation} \label{eq:pr_inequality}
    p^s (p-1) \ord_p(\alpha - \xi) > 1,
\end{equation}
and we set $r := \min(s,m)$, we see from \eqref{eq:alphan_only_padic} that 
%\DV{If $s  = r = 0$, but $m$ is non zero, is the second line below true?  I don't think you always get a strict inequality that allows you to use the non-archimedean property of the $2$-adic absolute value.  Right?  If $p=2$ and ${\rm ord}_{2}(\alpha - \xi) = 1$, then $s=0$, but if $\mu = -1$, then ${\rm ord}_{2}(1 - \mu) = 1$ as well.  So I still think we need our previous definition for $s$ and not this one.}
\begin{equation*}
\begin{aligned}
|\alpha^{n} - 1|_{p} &= \left( \prod_{\mu \in W_{p^r}} \lvert \alpha - \xi \mu \rvert_p \right) \left( \prod_{\mu \in W_{p^m} \setminus W_{p^r}} \lvert \alpha - \xi + \xi (1 - \mu) \rvert_p \right) \\ &= |\alpha - \xi|_{p} \left( \prod_{\mu \in W_{p^r} \setminus \{1\}} \frac{\lvert \alpha - \xi \mu \rvert_p}{|1 - \mu|_{p}} \, |1 - \mu|_{p} \right) \left( \prod_{\mu \in W_{p^m} \setminus W_{p^r}}|1 - \mu|_{p} \right) \\
&= |\alpha - \xi|_{p} \left( \prod_{\mu \in W_{p^r} \setminus \{1\}} \frac{\lvert \alpha - \xi \mu \rvert_p}{\lvert 1 - \mu \rvert_p} \right) \left( \prod_{\mu \in W_{p^m} \setminus \{1\}} |1 - \mu|_{p} \right) \\
&= |\alpha - \xi|_{p} \left( \prod_{\mu \in W_{p^r} \setminus \{1\}} \frac{\lvert \alpha - \xi \mu \rvert_p}{\lvert 1 - \mu \rvert_p} \right) |n|_{p},
\end{aligned}
\end{equation*}
where the last equality follows from the fact that 
%\DV{Is this true if $s<m$?  I think you want $m$ rather than $r$ in the line below?}
\[
\prod_{\mu \in W_{p^m} \setminus \{1\}} |1 - \mu|_{p} = \prod_{\mu \in W_n \setminus \{1\}} \lvert 1 - \mu \rvert_p = \left\lvert \mathrm{Res}\left( \frac{t^n - 1}{t-1}, t-1 \right) \right\rvert_p = \lvert n \rvert_p.
\]
Therefore, we see that $\mathrm{ord}_p(\alpha^n - 1) = \ord_p(n) + c(\mathrm{ord}_p(n))$, where the expression
\[
\begin{aligned}
    c(m) &:= \mathrm{ord}_p(\alpha - \xi) + \sum_{\mu \in W_{p^r} \setminus \{1\}} (\ord_p(\alpha - \xi \mu) - \ord_p(1 - \mu))\\
     &= \ord_p(\alpha^{p^r} - \xi^{p^r}) - r
\end{aligned}
\]
depends only on $m$ and $\alpha$, and is evidently constant when $m$ becomes sufficiently large.
This proves the second part of \eqref{stat1}, and allows us to conclude.
\end{proof}

\begin{remark}
    Let $\alpha \in \overline{\mathbb{Q}}$. Then, there exists a finite subset $S \subseteq \mathbb{N}$ such that for every rational prime $p \in \mathbb{N} \setminus S$ and every embedding $\iota \colon \mathbb{Q}(\alpha) \hookrightarrow \overline{\mathbb{Q}}_p$ we have that $\lvert \iota(\alpha) \rvert_p = 1$ and \[\ord_p(\iota(\alpha) - \tau(\pi(\iota(\alpha)))) \in \mathbb{N}.\] Therefore, for every rational prime $p \in \mathbb{N} \setminus (S \cup \{2\})$ we see that $s = 0$ is the minimal non-negative integer such that \eqref{eq:pr_inequality} holds true.
\end{remark}

From \cref{thm:PL_generic,thm:PL_unit_circle}, we obtain several corollaries. First of all, one can obtain an explicit formula for the $p$-adic valuation of the elements of the Pierce-Lehmer sequence associated to a polynomial $f \in \mathbb{Z}[t]$.
\begin{corollary} \label{cor:PL_p_adic_valuation}
Let $f \in \mathbb{Z}[t] \setminus \{0\}$ be a polynomial which does not vanish at any root of unity, and fix a prime $p$. Let $\beta_1,\dots,\beta_d \in \overline{\mathbb{Q}}_p$ denote the $p$-adic roots of $f$, counted with multiplicity. Then, for every $n \in \mathbb{N}$ we define
\begin{align}
\mu_p(f) &:= -m_p(f)/\log(p) \label{eq:mu_p} \\
B_{p,n}(f) &:= \{ \beta \in \overline{\mathbb{Q}}_p \colon f(\beta) = 0, \lvert \beta \rvert_p = 1, \lvert \beta^n - 1 \rvert_p < 1 \}, \label{eq:Bpn} \\
\lambda_{p,n}(f) &:= \#\{ j \in \{1,\dots,d\} \colon \beta_j \in B_{p,n}(f) \} \label{eq:lambdapn}
\end{align}
and for every $\beta \in \overline{\mathbb{Q}}_p$ such that $\lvert \beta \rvert_p = 1$ we set
\begin{equation} \label{eq:sp_beta}
s_p(\beta) := \min\{s \in \mathbb{Z}_{\geq 0} \colon p^s (p-1) \ord_p(\beta - \tau_p(\pi_p(\beta))) > 1 \}
\end{equation}
and we write $r_{p,n}(\beta) := \min(\ord_p(n),s_p(\beta))$.
Using this notation, we have that
\begin{equation} \label{eq:PL_p-adic_valuation}
\ord_p(\Delta_n(f)) = \mu_p(f) \cdot n + \lambda_{p,n}(f) \cdot \ord_p(n) + \nu_{p,n}(f),
\end{equation}
where $\nu_{p,n}(f) := \sum_{\substack{j \in \{1,\dots,d\} \\ \beta_j \in B_{p,n}(f)}} \left( \ord_p(\beta_j^{p^{r_{p,n}(\beta_j)}} - \tau_p(\pi_p(\beta_j))^{p^{r_{p,n}(\beta_j)}}) - r_{p,n}(\beta_j) \right)$.
\begin{comment}
Let $f \in \mathbb{Z}[t] \setminus \{0\}$. Then, for every prime $p$ and every $n \in \mathbb{N}$ we define
    \[
        \begin{aligned}
            \mu_p(f) &:= -m_p(f)/\log(p) \\
            B_{p,n}(f) &:= \{ \beta \in \overline{\mathbb{Q}}_p \colon f(\beta) = 0, \lvert \beta \rvert_p = 1, \lvert \beta^n - 1 \rvert_p < 1 \},
        \end{aligned}
    \]
    and for every $\beta \in B_{p,n}(f)$ we set
    \[
        s_p(\beta) := \min\{s \in \mathbb{N} \colon p^s (p-1) \ord_p(\beta - \tau_p(\pi_p(\beta))) > 1 \}
    \]
    and we write $r_{p,n}(\beta) := \min(\ord_p(n),s_p(\beta))$.
    Using this notation, we have that
    \[
        \ord_p(\Delta_n(f)) = \mu_p(f) \cdot n + \# B_{p,n}(f) \cdot \ord_p(n) + \nu_{p,n}(f),
    \]
    where $\nu_{p,n}(f) := \sum_{\beta \in B_{p,n}(f)} \left( \ord_p(\beta^{p^{r_{p,n}(\beta)}} - \tau_p(\pi_p(\beta))^{p^{r_{p,n}(\beta)}}) - r_{p,n}(\beta) \right)$.
\end{comment}
\end{corollary}
\begin{proof}
    We see from \cref{thm:PL_generic} that
    \[
        \ord_p(\Delta_n(f)) = \mu_p(f) \cdot n + \sum_{\substack{j \in \{1,\dots,d\} \\ \beta_j \in B_{p,n}(f)}} \ord_p(\beta_j^n - 1),
    \]
    and \cref{thm:PL_unit_circle} implies that
    \[
        \ord_p(\beta^n - 1) = \ord_p(n) + \ord_p(\beta^{p^{r_{p,n}(\beta)}} - \tau_p(\pi_p(\beta))^{p^{r_{p,n}(\beta)}}) - r_{p,n}(\beta)
    \]
    for every $\beta \in B_{p,n}(f)$, which allows us to conclude.
\end{proof}

Moreover, we can use \cref{thm:PL_generic,thm:PL_unit_circle} to pin down the asymptotic behaviour of the $p$-adic valuation of the Pierce-Lehmer sequence associated to an integral polynomial which does not vanish on the $p$-adic unit circle, or on roots of unity.

\begin{corollary} \label{cor:PL_p-adic_asymptotic}
Let $f(t) \in \mathbb{Z}[t] \setminus \{0\}$, and assume that $f(\alpha) \neq 0$ for every $\alpha \in \mathbb{C}_p$ such that $\lvert \alpha \rvert_p = 1$. Then, we have 
\begin{equation} \label{eq:p-adic_Pierce_Lehmer}
    |\Delta_{n}(f)|_{p} = M_{p}(f)^{n}
\end{equation}
for all $n \in \mathbb{N}$.  
If one only assumes that $f(\zeta) \neq 0$ for every $\zeta \in W_\infty$, one has
$$|\Delta_{n}(f)|_{p}^{1/n} \to M_{p}(f) $$
as $n \to \infty$.  
\end{corollary}
\begin{proof}
This follows directly from \cref{thm:PL_generic,thm:PL_unit_circle}. We leave the details to the reader.
\end{proof}

\begin{remark}
    A similar result holds true for the Archimedean Mahler measure.  
    More precisely, we see directly from the definition given in \eqref{eq:Archimedean_Mahler_measure} that for every polynomial $f \in \mathbb{Z}[t] \setminus \{0\}$ that does not vanish on the unit circle of $\mathbb{C}$, the asymptotic
    \begin{equation} \label{euclidean_growth_away_uc}
    |\Delta_{n}(f)|_\infty \sim M_{\infty}(f)^{n}
    \end{equation}
    holds true as $n \to +\infty$.
    If one only assumes that the roots of $f$ are not roots of unity, then one sees that
\begin{equation} \label{euclidean_growth}
|\Delta_{n}(f)|_\infty^{1/n} \to M_{\infty}(f),
\end{equation}
as $n \to \infty$. This follows from an inequality originally proved by Gelfand, as explained for instance in \cite[Lemma 1.10]{Everest/Ward:1999}.
On the other hand, if $f$ has some root on the unit circle of $\mathbb{C}$, the behaviour of the absolute values $\lvert \Delta_n(f) \rvert_\infty$ is quite chaotic, as exemplified for instance by \cite[Theorem~2.16]{Everest/Ward:1999}, which shows that the sequence of ratios $\lvert \Delta_n(f)/\Delta_{n-1}(f) \rvert_\infty$ converges if and only if $f$ has no roots on the unit circle of $\mathbb{C}$.
\end{remark} 

The $p$-adic valuation of various subsequences of a Pierce-Lehmer sequence can be understood from \cref{thm:PL_generic,thm:PL_unit_circle}. 
For instance, the following corollary shows that such a $p$-adic valuation of the sub-sequence $\{ \Delta_{p^n}(f) \}_{n=0}^{+\infty}$ associated to a polynomial $f \in \mathbb{Z}[t]$ that does not vanish at roots of unity exhibits a behavior similar to the $p$-adic valuation of the class number in $\mathbb{Z}_{p}$-extensions of number fields, which was already studied in the seminal work of Iwasawa \cite{Iwasawa:1959}.
\begin{comment}
The $p$-adic valuation of various subsequences of a Pierce-Lehmer sequence can be understood from \cref{thm:PL_generic,thm:PL_unit_circle}. 
For instance, the following corollary shows that such a $p$-adic valuation exhibits a behavior similar to the $p$-adic valuation of the class number in $\mathbb{Z}_{p}$-extensions of number fields.
\end{comment}

\begin{comment}
\begin{corollary} \label{cor:PL_Iwasawa}
Let $f(t) \in \mathbb{Z}[t] \setminus \{0\}$ be a polynomial which does not vanish at roots of unity. 
Then, for each rational prime $p$ there exist two constants $k_0 \in \mathbb{N}$ and $c \in \mathbb{Z}$, depending on $p$ and $f$, such that the following equality 
\[ {\rm ord}_{p}(\Delta_{p^{k}}(f)) = \mu_{p} \cdot  p^{k} + \lambda_{p} \cdot k + c \]
holds true for every $k \geq k_0(p)$, where 
$\mu_{p} := -m_{p}(f)/\log(p)$ and
\[\lambda_{p} := \# \{\alpha_{i} \, | \, {\rm ord}_{p}(\alpha_{i}) = 0 \text{ and } {\rm ord}_{p}(\alpha_{i} - 1) > 0 \}.\]
\end{corollary}
\begin{proof}
From \cref{thm:PL_generic}, we have
$${\rm ord}_{p}(\Delta_{n}(f)) = -\frac{m_{p}(f)}{\log p} \cdot n + \sum_{\substack{i=1 \\ |\alpha_{i}|_{p} = 1}}^{d} {\rm ord}_{p}(\alpha_{i}^{n} - 1). $$
Let $\alpha_{i}\in \overline{\mathbb{Q}}_{p}$ be a root satisfying $|\alpha_{i}|_{p}=1$.  Since $(N(\alpha_{i}),p)=1$, one has 
$${\rm ord}_{p}(\alpha_{i}^{p^{k}} - 1) = 0$$
unless $N(\alpha_{i}) = 1$ by \cref{thm:PL_unit_circle}. To conclude, it is sufficient to observe that $N(\alpha_{i}) = 1$ precisely when ${\rm ord}_{p}(\alpha_{i} - 1)> 0$.
\end{proof}
\end{comment}

\begin{corollary} \label{cor:PL_Iwasawa}
    Let $f(t) \in \mathbb{Z}[t]$ be a polynomial which does not vanish at roots of unity, $p$ be a rational prime and $\beta_1,\dots,\beta_d$ denote the $p$-adic roots of $f$, counted with multiplicity.
    Then, there exist two constants $k_0(f,p) \in \mathbb{N}$ and $\nu_p(f) \in \mathbb{Z}$, depending on $p$ and $f$, such that the following equality 
    \[ {\rm ord}_{p}(\Delta_{p^{k}}(f)) = \mu_{p}(f) \cdot  p^{k} + \lambda_{p}(f) \cdot k + \nu_p(f) \]
    holds true for every $k \geq k_0(f,p)$, where 
    $\mu_{p}(f) := -m_{p}(f)/\log(p)$ and
    \[\lambda_{p}(f) := \# \{j \in \{ 1,\dots,d \} \colon \lvert \beta_j \rvert_p = 1, \ \lvert \beta_j - 1 \rvert_p < 1 \}.\]
    \end{corollary} 
\begin{proof}
This follows directly from \cref{cor:PL_p_adic_valuation}. 
Indeed, the invariant $\mu_p(f)$ coincides with the one introduced in \eqref{eq:mu_p}. 
Moreover, let us note that $B_{p,p^k}(f) = B_{p,1}(f)$ for every $k \geq 1$. To see this, fix any $\beta \in \overline{\mathbb{Q}}_p$ such that $\lvert \beta \rvert_p = 1$. Then, we have that $\lvert \beta^{p^k} - 1 \rvert_p < 1$ if and only if the multiplicative order of $\pi_p(\beta) \in \overline{\mathbb{F}}_p^\times$ is a multiple of $p^k$. However, the aforementioned multiplicative order is coprime to $p$. Therefore, $\lvert \beta^{p^k} - 1 \rvert_p < 1$ if and only if $\pi_p(\beta) = 1$, which is equivalent to say that $\lvert \beta - 1 \rvert_p < 1$. Hence, we see immediately from the definition of the sets $B_{p,n}(f)$ given in \eqref{eq:Bpn} that $B_{p,p^k}(f) = B_{p,1}(f)$ for every $k \geq 1$, as we wanted to show.
This shows in particular that $\lambda_{p,p^k}(f) = \lambda_p(f)$ for every $k \geq 1$, where $\lambda_{p,p^k}(f)$ is the invariant defined in \eqref{eq:lambdapn}.

To conclude, it suffices to define
\[k_0(f,p) := \max\{ s_p(\beta) \colon \beta \in \overline{\mathbb{Q}}_p, \ \ f(\beta) = 0, \ \lvert \beta \rvert_p = 1 \},\]
where $s_p(\beta)$ is the invariant defined in \eqref{eq:sp_beta}.
Then, for every $k \geq k_0(f,p)$ and every $\beta \in \overline{\mathbb{Q}}_p$ such that $\lvert \beta \rvert_p = 1$ and $f(\beta) = 0$ we have that $r_{p,p^k}(\beta) = s_p(\beta)$, as follows immediately from the fact that $r_{p,p^k}(\beta) = \min(k,s_p(\beta))$. 
Therefore, using the definition of $\nu_{p,n}(f)$ given in \cref{cor:PL_p_adic_valuation}, we see that for every integer $k \geq k_0(f,p)$ the invariant
\[
\begin{aligned}
\nu_{p,p^k}(f) &= \sum_{\substack{j \in \{1,\dots,d\} \\ \beta_j \in B_{p,p^k}(f)}} \left( \ord_p(\beta_j^{p^{r_{p,p^k}(\beta_j)}} - \tau_p(\pi_p(\beta_j))^{p^{r_{p,p^k}(\beta_j)}}) - r_{p,p^k}(\beta_j) \right) \\ &= \sum_{\substack{j \in \{1,\dots,d\} \\ \beta_j \in B_{p,1}(f)}} \left( \ord_p(\beta_j^{p^{s_p(\beta_j)}} - \tau_p(\pi_p(\beta_j))^{p^{s_p(\beta_j)}}) - s_p(\beta_j) \right)
\end{aligned}
\]
is independent of $k$.
This allows us to set
\[
\nu_p(f) := \sum_{\substack{j \in \{1,\dots,d\} \\ \beta_j \in B_{p,1}(f)}} \left( \ord_p(\beta_j^{p^{s_p(\beta_j)}} - \tau_p(\pi_p(\beta_j))^{p^{s_p(\beta_j)}}) - s_p(\beta_j) \right),
\]
and to conclude our proof.

\end{proof}

On the other hand, if $\ell, p \in \mathbb{N}$ are two distinct rational primes, the $\ell$-adic valuation of the sub-sequence $\{ \Delta_{p^n}(P) \}_{n=0}^{+\infty}$ exhibits a behavior which is similar to the one observed by Washington \cite{Washington:1978} in the case of $\mathbb{Z}_p$-towers of number fields.

\begin{comment}
\begin{corollary} \label{cor:PL_Washington}
Let $f(t) \in \mathbb{Z}[t] \setminus \{0\}$ be a polynomial which does not vanish at roots of unity. 
Let $p$ and $\ell$ be two distinct rational primes. 
Then, there exist two constants $k_0 \in \mathbb{N}$ and $c \in \mathbb{Z}$, depending on $p$, $\ell$ and $f$, such that the equality
$${\rm ord}_{p}(\Delta_{\ell^{k}}(f)) = \mu_{p} \cdot  \ell^{k} +  c, $$
holds true for every $k \geq k_0$, where again
$\mu_{p} = -m_{p}(f)/\log(p)$.
\end{corollary}
\end{comment}
\begin{corollary} \label{cor:PL_Washington}
    Let $f(t) \in \mathbb{Z}[t]$ be a polynomial which does not vanish at roots of unity. 
    Let $p$ and $\ell$ be two distinct rational primes. 
    Then, there exist two constants $k_0(f,p,\ell) \in \mathbb{N}$ and $\nu_p(f,\ell) \in \mathbb{Z}$, depending on $p$, $\ell$ and $f$, such that the equality
    \[{\rm ord}_{p}(\Delta_{\ell^{k}}(f)) = \mu_{p}(f) \cdot \ell^{k} +  \nu_p(f,\ell) \]
    holds true for every $k \geq k_0(f,p,\ell)$, where again
    $\mu_{p}(f) = -m_{p}(f)/\log(p)$.
    \end{corollary}
    \begin{proof}
    This follows once again from \cref{cor:PL_p_adic_valuation}. Indeed, $\mu_p(f)$ is once again identical to the invariant defined in \eqref{eq:mu_p}. Moreover, $\mathrm{ord}_p(\ell^k) = 0$ for every $k \geq 0$, which shows that the part of the equality \eqref{eq:PL_p-adic_valuation} involving the invariant $\lambda_{p,n}(f)$ does not appear when $n = \ell^k$. 
    
    To conclude, let $\beta_1,\dots,\beta_d \in \overline{\mathbb{Q}}_p$ be the $p$-adic roots of $f(x)$, counted with multiplicity, that lie on the $p$-adic unit circle, and $N_1,\dots,N_d$ be the multiplicative orders of their reductions 
    \[\pi_p(\beta_1),\dots,\pi_p(\beta_d) \in \overline{\mathbb{F}}_p^\times.\]
    Then, for every $j \in \{1,\dots,d\}$ we have that $\lvert \beta_j^{\ell^k} - 1 \rvert_p < 1$ if and only if there exists some non-negative integer $a_j \le k$ such that $N_j = \ell^{a_j}$. Moreover, we have that \[r_{p,\ell^k}(\beta_j) = \min(\mathrm{ord}_p(\ell^k),s_p(\beta_j)) = 0\] for every $j \in \{1,\dots,d\}$ and every $k \geq 0$.
    Therefore, if we set 
    \[
    \begin{aligned}
        k_0(f,p,\ell) &:= \max\{\mathrm{ord}_\ell(N_j) \colon j \in \{ 1,\dots,d \} \} \\
        \nu_p(f,\ell) &:= \sum_{\substack{ j \in \{1,\dots,d\} \\ \beta_j \in B_{p,\ell^{k_0}}(f)}} \left( \mathrm{ord}_p(\beta_j - \tau_p(\pi_p(\beta_j))) \right),
    \end{aligned}
    \]
    where we write $k_0$ instead of $k_0(f,p,\ell)$.
    Then, we see from the definition of $\nu_{p,n}(f)$ given in \cref{cor:PL_p_adic_valuation} that $\nu_{p,\ell^k}(f) = \nu_p(f,\ell)$ for every $k \geq k_0(f,p,\ell)$, and this allows us to conclude.
    \end{proof}

In fact, \cref{cor:PL_Iwasawa,cor:PL_Washington} can be generalized by looking at sequences of integers which are divisible only by a finite number of primes, as done by Friedman \cite{Friedman:1982} for cyclotomic extensions of number fields which are abelian over $\mathbb{Q}$.

\begin{corollary}
    \label{cor:PL_Friedman}
    Let $f(t) \in \mathbb{Z}[t]$ be a polynomial which does not vanish at roots of unity. Let $\ell_1,\dots,\ell_r$ be distinct prime numbers, and let $\mathcal{S} \subseteq \mathbb{N}$ be the set of those integers whose prime divisors are contained in $\{\ell_1,\dots,\ell_r\}$. Then:
    \begin{itemize}
        \item for every $j \in \{1,\dots,r\}$ there exist a non-negative integer $\lambda_j(f)$ and an integer $\nu_j(f)$ such that for every $n = \ell_1^{k_1} \cdots \ell_j^{k_j} \cdots \ell_r^{k_r} \in \mathcal{S}$ we have that
        \[
            \mathrm{ord}_{\ell_j}(\kappa(X_n)) = \mu_{\ell_j}(f) \cdot n + \lambda_j(f) \cdot k_j + \nu_j(f)
        \]
        as long as $n$ is big enough, where again $\mu_{\ell_j}(f) = -m_{\ell_j}(f)/\log(\ell_j)$.
        \item for every prime $p \not\in \{\ell_1,\dots,\ell_r\}$, there exists an integer $\nu(f)$ such that
        \[
            \mathrm{ord}_p(\kappa(X_n)) = \mu_p(f) \cdot n + \nu(f),
        \]
        for every $n \in \mathcal{S}$ which is big enough.
    \end{itemize}
\end{corollary}
\begin{proof}
    The proof is similar to the proofs of the two previous corollaries, and we leave it to the reader.
\end{proof}

\begin{remark}
    In the situation where the polynomial $f(t)$ is monic, the $p$-adic valuation of a Pierce-Lehmer sequence was also studied in \cite{Ji/Qin:2015}.  In this situation, there is no $p$-adic Mahler measure appearing in the formulas.
\end{remark}

\begin{remark}
    Note that the Pierce-Lehmer sequence $\{ \Delta_n(f) \}_{n \in \mathbb{N}}$ associated to any polynomial $f \in \mathbb{Z}[t]$ satisfies a linear recurrence, as explained in \cite[Section~8]{Lehmer_1933}.
    Therefore, studying the $p$-adic valuation of Pierce-Lehmer sequences can be seen as a special case of the more general problem of studying the $p$-adic valuation of linearly recurrent sequences, which has been the subject of great attention (see for instance \cite{Bilu_Luca_Nieuwveld_Ouaknine_Worrell_2023}).
    We also refer the interested reader to the works \cite{EEW:2000,Flatters_2009}, which treat problems related to the $p$-adic valuation of Pierce-Lehmer sequences.
\end{remark}

%\section{Graph theory}
\section{Graph theory} \label{GT}

The aim of this section is to prove \cref{p_adic_val}, which provides an explicit expression for the $p$-adic valuation of the number of spanning trees in a $\mathbb{Z}$-tower of graphs in terms of a polynomial naturally associated to this tower, which we call the \emph{Ihara polynomial} and which we define in \cref{growth}.
In particular, \cref{thm:main_thm_intro} is a simplified version of \cref{p_adic_val}, as we explain in \cref{sec:exact_formulas}.
To do so, we first recall some fundamentals about graphs and their covers in \cref{loc_fin_graph,number}.
Then, we devote \cref{growth} to the proof of \cref{thm:kN_Pierce_Lehmer}, which provides an explicit formula relating the number of spanning trees of the members of a $\mathbb{Z}$-cover of graphs to the Pierce-Lehmer sequence associated to the Ihara polynomial of this tower.
We provide an explicit example which verifies this relation in 
\cref{sec:explicit_example_bouquet}.
Moreover, \cref{sec:asymptotics} shows how to combine \cref{thm:kN_Pierce_Lehmer} with the results proven in \cref{sec:Pierce_Lehmer}, to provide some asymptotic expressions for the growth of the number of spanning trees in a $\mathbb{Z}$-tower.
In particular, this generalizes two previous results of A.D. Mednykh and I.A. Mednykh \cite{Mednykh:2018,Mednykh:2019}.

%\subsection{Locally finite graphs}
\subsection{Galois covers of locally finite graphs} \label{loc_fin_graph}
The aim of this sub-section is to formally introduce the kind of graphs which are considered in this article, and their Galois theory.

\subsubsection*{Locally finite graphs} Let $X = (V_{X},\mathbf{E}_{X})$ be a graph in the sense of Serre (see \cite{Serre:1977} and also \cite{Sunada:2013}), where $V_X$ and $\mathbf{E}_X$ are two sets, to be interpreted as the sets of vertices and edges of $X$. In particular, each edge $e \in \mathbf{E}_X$ has an origin $o(e) \in V_X$ and a terminus $t(e) \in V_X$, giving rise to the incidence map 
\[
	\begin{aligned}
		\mathrm{inc} \colon \mathbf{E}_{X} &\to V_{X} \times V_{X} \\
		e &\mapsto (o(e),t(e)),
	\end{aligned}
\]
and to the inversion map $\mathbf{E}_{X} \rightarrow \mathbf{E}_{X}$, denoted by $e \mapsto \bar{e}$, such that \[(o(\overline{e}),t(\overline{e})) = (t(e),o(e))\] and $\overline{\overline{e}} = e \neq \overline{e}$ for every $e \in \mathbf{E}_X$. 

A graph $X$ is called finite if both $V_{X}$ and $\mathbf{E}_{X}$ are finite sets. Moreover, a graph $X$ is called locally finite if for each vertex $v \in V_{X}$, the set of edges with origin at $v$, defined as
\[
\mathbf{E}_{X,v} = \{ e \in \mathbf{E}_{X} \, | \, o(e) = v\}
\]
is finite. In this case, one defines the valency (or degree) of a vertex $v \in V_{X}$ to be
$${\rm val}_{X}(v) = |\mathbf{E}_{X,v}|. $$
Any finite graph is in particular locally finite. 

\begin{assumption}
In this paper, all graphs will be locally finite.    
\end{assumption}

\subsubsection*{Paths and loops}
Let us recall that a path $c = e_{1} \cdot \ldots \cdot e_{m}$ in a graph $X$ consists of a sequence of directed edges $e_{i} \in \mathbf{E}_{X}$ such that $t(e_{i}) = o(e_{i+1})$ for each index $i \in \{1,\ldots,m-1\}$. 
The origin and the terminus of the path $c = e_1 \cdot \ldots \cdot e_m$ are defined as $o(c) = o(e_{1})$ and $t(c) = t(e_{m})$ respectively. 

A graph $X$ is called connected if given any two vertices $v_{1}, v_{2} \in V_{X}$, there exists a path $c$ in $X$ such that $o(c) = v_{1}$ and $t(c) = v_{2}$.
Finally, a loop based at a vertex $v_0 \in X$ is a path $c$ in $X$ such that $o(c) = t(c) = v_0$. 

This allows one to define the fundamental group of $X$ based at a vertex $v_0 \in V_X$, which is denoted by $\pi_1(X,v_0)$, as the set of loops based at $v_0$, considered modulo homotopy (see \cite[Section~3.5]{Sunada:2013} for the precise definition of this equivalence relation in the context of graphs), endowed with the group operation given by the concatenation of paths (see \cite[Section~5.3]{Sunada:2013} for details).

\subsubsection*{Galois covers of graphs}
Let $Y$ and $X$ be two graphs.  A morphism of graphs $f:Y \rightarrow X$ is called a cover (or a covering map) if the following two conditions are satisfied:
\begin{enumerate}
\item $f:V_{Y} \rightarrow V_{X}$ is surjective,
\item for all $w \in V_{Y}$, the restriction $f|_{\mathbf{E}_{Y,w}}$ induces a bijection
$$f|_{\mathbf{E}_{Y,w}}:\mathbf{E}_{Y,w} \stackrel{\approx}{\rightarrow} \mathbf{E}_{X,f(w)}. $$
\end{enumerate}
We will often refer to $Y/X$ as a cover if the covering map is understood from the context.  Given a cover $f:Y \rightarrow X$, one defines as usual ${\rm Aut}_{f}(Y/X)$ to be the subgroup of ${\rm Aut}(Y)$ consisting of the automorphisms $\iota \in {\rm Aut}(Y)$ satisfying $f \circ \iota = f$.  Again, we will often write ${\rm Aut}(Y/X)$ instead of ${\rm Aut}_{f}(Y/X)$ if $f$ is understood.  

Let us also recall that a cover $f:Y \rightarrow X$ is called Galois if the following two conditions are satisfied:
\begin{enumerate}
\item the graph $Y$ is connected (and hence also $X$),
\item the group ${\rm Aut}(Y/X)$ acts transitively on the fiber $f^{-1}(v)$ for all $v \in V_{X}$.
\end{enumerate}
If $Y/X$ is a Galois cover, we write ${\rm Gal}(Y/X)$ instead of ${\rm Aut}(Y/X)$.  In this case, one has the usual Galois correspondence between subgroups of ${\rm Aut}(Y/X)$ and equivalence classes of intermediate covers of $Y/X$.  

\subsubsection*{Voltage assignments}
Let $X$ be a graph and let $G$ be a group.  A voltage assignment on $X$ with values in $G$ is defined to be a function $\alpha:\mathbf{E}_{X} \rightarrow G$ satisfying
\begin{equation} \label{pr}
\alpha(\bar{e}) = \alpha(e)^{-1}
\end{equation}
for every $e \in \mathbf{E}_{X}$. 
Each such voltage assignment can be defined starting from an orientation of $X$, which is a subset $S \subseteq \mathbf{E}_X$ such that for each edge $e \in \mathbf{E}_X$, either $e$ or $\overline{e}$ belong to $S$, but not both.
Then, to get a voltage assignment as above, it suffices to define it on any orientation $S$ and set $\alpha(\overline{s}) := \alpha(s)^{-1}$ for every $s \in S$.

\subsubsection*{Covers from voltage assignments}
Given a graph $X$, a group $G$ and a voltage assignment
\[\alpha \colon \mathbf{E}_X \to G,\] 
one can construct a new graph $X(G,\alpha)$ as follows:
\begin{itemize}
    \item the vertices of $X(G,\alpha)$ are given by $V_X \times G$;
    \item the edges of $X(G,\alpha)$ are given by $\mathbf{E}_X \times G$;
    \item the origin, terminus and inverse maps are defined as:
\[
\begin{aligned}
  o(e,\sigma) &= (o(e),\sigma) \\ 
  t(e,\sigma) &= (t(e),\sigma \cdot \alpha(e)) \\
  \overline{(e,\sigma)} &= (\bar{e},\sigma \cdot \alpha(e))
\end{aligned}
\]
for each edge $(e,\sigma) \in \mathbf{E}_X \times G$.
\end{itemize}
It is easy to see that if $X$ is locally finite than so is $X(G,\alpha)$, and that the map of graphs 
\[p \colon X(G,\alpha) \rightarrow X\] 
defined as $p(v,\sigma) := v$ on each vertex $(v,\sigma) \in V_X \times G$, and as $p(e,\sigma) := e$ on each edge $(e,\sigma) \in \mathbf{E}_X \times G$, is actually a covering map. Moreover, this covering map is Galois whenever $X(G,\alpha)$ is connected, and in this case $\operatorname{Gal}(X(G,\alpha)/X) \cong G$ canonically.

To conclude, let us observe that the construction of $X(G,\alpha)$ is functorial with respect to $\alpha$. More precisely, for each morphism of groups $f \colon G \to H$ one gets a new voltage assignment $\beta := f \circ \alpha$ with values in $H$, and a natural map of graphs 
\begin{equation} \label{eq:functorial_Galois_cover}
    f_\ast \colon X(G,\alpha) \to X(H,\beta),
\end{equation} 
which is defined on each vertex $(v,\sigma) \in V_X \times G$ as $f_{\ast}(v,\sigma) := (v,f(\sigma))$ , and on each edge $(e,\sigma) \in \mathbf{E}_X \times G$ as $f_\ast(e,\sigma) := (e,f(\sigma))$. 
Finally, it is easy to see that this morphism $f_\ast$ is a covering map whenever $f$ is surjective.

\subsubsection*{Monodromy representations}
Let $X$ be a graph and $\alpha \colon \mathbf{E}_X \to G$ be a voltage assignment with values in a group $G$. Then, the monodromy representation attached to $\alpha$ at a vertex $v_0 \in V_X$ is given by the following map
\begin{equation} \label{pi_gm}
\begin{aligned}
    \rho_{\alpha,v_0} \colon \pi_{1}(X,v_{0}) &\to G \\
    [e_1 \cdot \ldots \cdot e_n] &\mapsto \alpha(e_1) \cdot \ldots \cdot \alpha(e_n)
\end{aligned}
\end{equation}
which is easily seen to be a well-defined morphism of groups.
Moreover, this map can be used to detect when the graph $X(G,\alpha)$ is connected, and thus when the natural covering map \[p \colon X(G,\alpha) \to X\] is Galois, as we recall in the following theorem, which is proven in \cite[Section~2.3.1]{Ray/Vallieres:2022}.

\begin{theorem} \label{connectedness}
Let $X$ be a connected graph, and $\alpha \colon \mathbf{E}_X \to G$ be a voltage assignment. Then, the graph $X(G,\alpha)$ is connected if and only if the monodromy representation $\rho_{\alpha,v_0}$ attached to $\alpha$ at some (equivalently, any) vertex $v_0 \in V_X$ is surjective.
\end{theorem}

\begin{remark} \label{rmk:universal_cover}
Let $X$ be a connected graph, $v_0 \in V_X$ a vertex of $X$, and $\alpha \colon \mathbf{E}_X \to G$ a voltage assignment such that $\rho_{\alpha,v_0}$ is surjective.
Fix moreover a universal cover $\pi \colon \widetilde{X} \twoheadrightarrow X$, and a vertex $w_0 \in \widetilde{X}$ such that $\pi(w_0) = v_0$.
Thanks to the universal property of the universal cover, proved for example in \cite[Theorem~5.10]{Sunada:2013}, it can be shown that the intermediate Galois cover of $\widetilde{X} \to X$ given by $X(G,\alpha)$ corresponds to the subgroup $\varphi_{w_0}(\ker(\rho_{\alpha,v_0})) \trianglelefteq {\rm Gal}(\widetilde{X}/X)$, where \[\varphi_{w_{0}}: \pi_{1}(X,v_{0}) \stackrel{\sim}{\longrightarrow}{\rm Gal}(\widetilde{X}/X)\] is the usual group isomorphism.
\begin{comment}
Let $X$ be a connected graph, $v_0 \in V_X$ a vertex of $X$, and $\alpha \colon \mathbf{E}_X \to G$ a voltage assignment such that $\rho_{\alpha,v_0}$ is surjective.
Fix moreover a universal cover $\pi \colon \widetilde{X} \twoheadrightarrow X$, and a point $w_0 \in \widetilde{X}$ such that $\pi(w_0) = v_0$.
Then, it can be shown that the intermediate Galois cover of $\widetilde{X} \to X$ given by $X(G,\alpha)$ corresponds to the subgroup $\varphi_{w_0}(\ker(\rho_{\alpha,v_0})) \trianglelefteq {\rm Gal}(\widetilde{X}/X)$, where $\varphi_{w_{0}}: \pi_{1}(X,v_{0}) \stackrel{\sim}{\longrightarrow}{\rm Gal}(\widetilde{X}/X)$ is the usual group isomorphism.
\end{comment}
\end{remark}

\subsubsection*{Systems of Galois covers}
Let $X$ be a graph and $\alpha \colon \mathbf{E}_X \to G$ a voltage assignment taking values in a group $G$. Given a group homomorphism $f \colon G \to H$ and a vertex $v_0 \in V_X$, the monodromy representation attached at $v_0$ to the Galois cover $f_\ast$ defined in \eqref{eq:functorial_Galois_cover} is given by $f \circ \rho_{\alpha,v_0}$.
This shows in particular that if the graph $X(G,\alpha)$ is connected, the graph $X(H,f \circ \alpha)$ will be connected whenever $f$ is surjective.

Moreover, any morphism of groups $f \colon G \to H$ induces another morphism of groups
\begin{equation} \label{eq:morphism_of_groups}
    {\rm ker}(f) \rightarrow {\rm Aut}(X(G,\alpha)/X(H,f \circ \alpha))
\end{equation}
which sends each $\tau \in \ker(f)$ to the automorphism $\phi_\tau \colon X(G,\alpha) \to X(G,\alpha)$ defined by setting $\phi_\tau(v,\sigma) := (v,\tau \cdot \sigma)$ for each vertex $(v,\sigma) \in V_X \times G$, and $\phi_\tau(e,\sigma) := (e,\tau \cdot \sigma)$ for each edge $(e,\sigma) \in \mathbf{E}_X \times G$. 
The morphism of groups \eqref{eq:morphism_of_groups} is actually an isomorphism whenever $f$ is surjective, as follows from the unique lifting theorem \cite[Theorem~5.1]{Sunada:2013}.

In particular, the previous discussion shows that any voltage assignment \[\alpha \colon \mathbf{E}_X \to G\] induces a system of Galois covers indexed by the lattice of quotients of the group $G$. As we will see in the upcoming sections of this paper, it is interesting to study how various graph invariants evolve when moving across this system.

\subsection{The number of spanning trees in finite abelian covers of finite graphs} \label{number}
One particularly interesting kind of invariant of a connected finite graph $X$ is given by its Picard group $\mathrm{Pic}^{0}(X)$, also known as the Jacobian, sandpile or class group of $X$. Its cardinality, denoted by $\kappa(X)$, is given by the number of spanning trees of the graph $X$. The aim of this section is to recall, following \cite[Section~3]{Vallieres:2021}, how this number changes in an abelian cover of a finite graph, using Ihara's determinant formula.

\subsubsection*{Ihara zeta functions}
To do so, we will make use of another invariant of a finite connected graph $X$, namely its Ihara zeta function, which we denote by $Z_X(u)$. This is a rational function of $u$, which can be explicitly computed thanks to the Ihara determinant formula, which we recall below in \eqref{eq:Ihara_determinant}, and is proven in \cite{Kotani/Sunada:2000} and \cite{Bass:1992}. 
More precisely, given an ordering $V_X = \{v_1,\dots,v_g\}$ of the vertices of $X$, we let $A_X := (a_{i,j}) \in \mathbb{Z}^{g \times g}$ denote the adjacency matrix of $X$, which is defined by setting $a_{i,j} := \# \{ e \in \mathbf{E}_X \colon o(e) = v_i, \ t(e) = v_j \}$. Moreover, we let $D_X := (d_{i,j}) \in \mathbb{Z}^{g \times g}$ denote the valency (or degree) matrix of $X$, which is a diagonal matrix defined by setting $d_{i,i} := \mathrm{val}_X(v_i)$ for each $i \in \{1,\dots,g\}$. 
Then, we can write the Ihara zeta function $Z_X(u)$ using the following explicit formula: 
\begin{equation} \label{eq:Ihara_determinant}
    Z_{X}(u) = \frac{1}{(1-u^{2})^{-\chi(X)} \cdot {\rm det}(\mathrm{Id}_g - A_X u + (D_X - \mathrm{Id}_g)u^{2})},
\end{equation}
where $\mathrm{Id}_g$ denotes the $g \times g$ identity matrix, and $\chi(X) := \lvert V_X \rvert - \lvert \mathbf{E}_X \rvert/2$ is the Euler characteristic of $X$.
In particular, we have that $Z_X(u)^{-1} = (1-u^2)^{-\chi(X)} \cdot h_X(u)$, where
\[
h_{X}(u) := {\rm det}(\mathrm{Id}_g - A_X u + (D_X - \mathrm{Id}_g)u^{2}) \in \mathbb{Z}[u]
\]
is a polynomial.

\subsubsection*{Hashimoto's formula}
This explicit formula can be used to relate the Ihara zeta function to the number of spanning trees of $X$. More precisely, given a finite connected graph $X$, one has
\begin{equation} \label{eq:Hashimoto}
    h_{X}'(1) = -2 \chi(X) \kappa(X),
\end{equation}
as was proven by Hashimoto in \cite[Theorem~B]{Hashimoto:1990} (see also \cite[Part~II, Sections~5~and~6]{Bass:1992}).
\begin{comment}
as was proven by Hashimoto in \cite{Hashimoto:1990} (see also \cite{Bass:1992}).
\end{comment}
Such a formula, which can be considered as an analogue of the class number formula in the context of graph theory, admits an equivariant generalization.

\subsubsection*{Artin-Ihara $L$-functions} 
Given a Galois cover of finite connected graphs $Y/X$, one can associate to any linear complex representation $\rho \colon \mathrm{Gal}(Y/X) \to \mathrm{GL}_n(\mathbb{C})$
an Artin-Ihara $L$-function $L_{Y/X}(u,\rho)$. 
This admits an explicit description analogous to \eqref{eq:Ihara_determinant}. 
More precisely, let $d_\rho \in \mathbb{N}$ denote the degree of the representation $\rho$, and fix an ordering $V_X = \{v_1,\dots,v_g\}$ of the vertices of $X$, and a section $\iota \colon V_{X} \rightarrow V_{Y}$ of the projection $V_{Y} \to V_{X}$.  Then, \cite[Theorem~18.15]{Terras:2011} shows that the Artin-Ihara $L$-function $L_{Y/X}(u,\rho)$ can be explicitly computed thanks to the following formula:
\begin{equation*}
    L_{Y/X}(u,\rho) = \frac{1}{(1-u^2)^{- \chi(X) \cdot d_\rho} \cdot \det(\mathrm{Id}_{gd_{\rho}} - A_\rho u + Q_\rho u^2)},
\end{equation*}
where $A_\rho, Q_\rho \in \mathbb{C}^{g d_\rho \times g d_\rho}$ are two explicit matrices, whose definition we now recall.
Given $\sigma \in G$ we let \[A(\sigma) := (\#\{ e \in \mathbf{E}_Y \colon o(e) = \iota(v_i), \ t(e) = \sigma(\iota(v_j)) \})_{i,j = 1,\dots,g},\] 
and we define
\begin{equation*}
	A_\rho := \sum_{\sigma \in G} A(\sigma) \otimes \rho(\sigma) \quad \text{and} \quad Q_\rho := (D_X \otimes \mathrm{Id}_{d_\rho}) - \mathrm{Id}_{g \cdot d_\rho},
\end{equation*}
where $\otimes$ denotes the Kronecker product of matrices.
For more details, we refer the interested reader to \cite[Definition~18.13]{Terras:2011}. 
\begin{comment}
whose definitions can be found in \cite[Definition~18.13]{Terras:2011}. 
\end{comment}
As before, this explicit formula allows one to write \[L_{Y/X}(u,\rho)^{-1} = (1-u^2)^{- \chi(X) \cdot d_\rho} \cdot h_{Y/X}(u,\rho),\] where
\[
    h_{Y/X}(u,\rho) := \det(\mathrm{Id}_{gd_{\rho}} - A_\rho u + Q_\rho u^2) \in \mathbb{C}[u]
\]
is a polynomial.
In particular, if $\mathrm{Gal}(Y/X)$ is abelian and $\psi$ is a character of $\mathrm{Gal}(Y/X)$, we have that $L_{Y/X}(u,\psi)^{-1} = (1-u^2)^{-\chi(X)} \cdot h_{Y/X}(u,\psi)$, where
\begin{equation} \label{eq:h_polynomial}
    h_{Y/X}(u,\psi) := \det(\mathrm{Id}_{g} - A_\psi u + (D_X - \mathrm{Id}_g) u^2),
\end{equation}
because $d_\psi = 1$ and $Q_\psi = D_X - \mathrm{Id}_g$, as follows easily from \cite[Definition~18.13]{Terras:2011}.

\subsubsection*{Spanning trees and abelian covers}
To conclude this sub-section, let us recall that the Artin-Ihara $L$-functions satisfy the Artin formalism (see \cite{Bass:1992} and \cite{Stark/Terras:2000}). 
This implies that for every Galois cover of finite graphs $Y/X$, with Galois group $G := \mathrm{Gal}(Y/X)$, the Ihara zeta function $Z_Y(u)$ admits the following factorization:
\[
    Z_Y(u) = \prod_{\rho \in \mathrm{Irr}(G)} L_{Y/X}(u,\rho)^{d_\rho},
\]
where $\mathrm{Irr}(G)$ denotes the set of equivalence classes of complex irreducible representations of a finite group $G$ (see \cite[Corollary~18.11]{Terras:2011}). 
Therefore, we see easily that
\begin{equation} \label{eq:h_factorization}
    h_{Y}(u) = \prod_{\rho \in \mathrm{Irr}(G)} h_{Y/X}(u,\rho)^{d_\rho},
\end{equation}
using the relation $\chi(Y) = \lvert G \rvert \cdot \chi(X)$ between the Euler characteristics of $Y$ and $X$, which is explained at the end of \cite[Section~5.1]{Sunada:2013}, and the classical identity $\lvert G \rvert = \sum_{\rho \in \mathrm{Irr}(G)} d_\rho^2$, proven for example in \cite[Section~2.4, Corollary~2]{Serre:1977_Reps}.
\begin{comment}
using the relation $\chi(Y) = \lvert G \rvert \cdot \chi(X) = \left( 
\sum_{\rho \in \mathrm{Irr}(G)} d_\rho^2 \right) \chi(X)$ between the Euler characteristics of $Y$ and $X$, which is explained in \cite[Page~55]{Sunada:2013}.
\end{comment}
Finally, if $\rho_0$ denotes the trivial representation of $G$, and $\chi(X) \neq 0$, we have that
\[
\lvert G \rvert \cdot \kappa(Y) = \kappa(X) \cdot \prod_{\rho \neq \rho_0} h_{Y/X}(1,\rho)^{d_\rho},
\]
thanks to the formulas \eqref{eq:Hashimoto} and \eqref{eq:h_factorization}, combined with the fact that 
\[
h_{Y/X}(1,\rho_0) = h_{X}(1) = 0,
\] 
which holds because the Laplacian matrix $D_X - A_X$ is singular (as explained in \cite[Proposition~2.8]{Corry_Perkinson_2018}).
In particular, if $Y/X$ is a Galois cover of finite graphs such that $\chi(X) \neq 0$ and $G := \mathrm{Gal}(Y/X)$ is abelian, we have
\begin{equation} \label{cnf}
|G| \cdot \kappa(Y) = \kappa(X) \prod_{\psi \neq \psi_{0}}h_{Y/X}(1,\psi),
\end{equation}
where $\psi_0$ denotes the trivial character of $G$, and $\psi \in G^\vee$ runs over all non-trivial characters of $G$.

\subsection{Exact formulas for the number of spanning trees} \label{growth}
In this sub-section, we introduce what we take the liberty to call the \emph{Ihara polynomial} $\mathcal{I}_{\alpha}$ associated to a voltage assignment $\alpha:\mathbf{E}_{X} \rightarrow G$.  
When $G = \mathbb{Z}^d$ for some $d \in \mathbb{N}$, this polynomial was introduced, with a slightly different terminology, in the work of Silver and Williams \cite{Silver_Williams_2021} (see \cref{rmk:Silver_Williams} for a comparison between the two definitions). 
Moreover, when $G \in \{ \mathbb{Z}, \mathbb{Z}_\ell \}$, for some rational prime $\ell$, this polynomial was considered by Lei and the second author of the present paper \cite{Lei/Vallieres:2023}. 
In the general case, this invariant consists of an element of the group ring $\mathbb{Z}[G]$, which we write as a generalized polynomial ring $\mathbb{Z}[t^G]$ by adding a formal variable $t$.  

\subsubsection*{The Ihara polynomial}
More precisely, let $X$ be a finite connected graph such that $\chi(X) \neq 0$, and let us start with a voltage assignment $\alpha:\mathbf{E}_{X} \rightarrow G$.  
As before, let us fix an ordering of the vertices of $X$, given by $V_{X} = \{v_{1},\ldots,v_{g}\}$. Then, we can define the matrix
\[ 
A_{\alpha}(t) :=  \left(\sum_{\substack{e \in \mathbf{E}_{X} \\ {\rm inc}(e) = (v_{i},v_{j})}} t^{\alpha(e)}\right) \in \mathbb{Z}[t^G]^{g \times g},
\]
which we use to introduce the \emph{Ihara polynomial}
\begin{equation} \label{eq:Ihara_polynomial}
    \mathcal{I}_{\alpha}(t) = {\rm det}(D_X - A_{\alpha}(t)) \in \mathbb{Z}[t^{G}],
\end{equation}
where $D_X$ denotes, as before, the valency (or degree) matrix of $X$.  When no confusion seems to occur, we will just write $\mathcal{I}_\alpha$ instead of $\mathcal{I}_\alpha(t)$.

We note in particular that the Ihara polynomial is self-reciprocal. In other words, we have the following identity:
\begin{equation} \label{f_eq}
\mathcal{I}_{\alpha}\left( \frac{1}{t}\right) = \mathcal{I}_{\alpha}(t),
\end{equation}
which comes from the fact that the transpose of the matrix $A_\alpha(t)$ equals the matrix $A_\alpha(t^{-1})$ by definition.
Moreover, for every morphism of groups $f \colon G \to H$ we have by definition that
\begin{equation} \label{eq:functorial_Ihara}
    \mathcal{I}_\beta = f_\ast(\mathcal{I}_\alpha),
\end{equation}
where $\beta := f \circ \alpha$ and $f_\ast \colon \mathbb{Z}[t^G] \to \mathbb{Z}[t^H]$ denotes the morphism of rings induced by $f$.

\begin{remark} \label{rmk:Silver_Williams}
    For $G = \mathbb{Z}^d$ with $d \in \mathbb{N}$, this polynomial was introduced in \cite{Silver_Williams_2021} under the name of Laplacian polynomial.
In particular, an unsigned $d$-periodic graph $\mathfrak{X}$ in the sense of \cite[Section~6]{Silver_Williams_2021} can be obtained as $\mathfrak{X} = X(\mathbb{Z}^d,\alpha)$, where $X$ is a finite graph and $\alpha \colon \mathbf{E}_X \to \mathbb{Z}^d$ is a voltage assignment. 
\end{remark}

\subsubsection*{The Ihara polynomial and the number of spanning trees}

Suppose now that $X$ is a finite connected graph such that $\chi(X) \neq 0$, and fix a voltage assignment \[\alpha \colon \mathbf{E}_X \to G\] with values in some finite abelian group $G$, such that the induced graph $X(G,\alpha)$ is connected. The following result expresses how the number of spanning trees changes from $X$ to $X(G,\alpha)$, using the Ihara polynomial $\mathcal{I}_\alpha$.

\begin{proposition} \label{prop:Ihara_spanning_general}
    For every finite connected graph $X$ such that $\chi(X) \neq 0$, and every voltage assignment $\alpha \colon \mathbf{E}_X \to G$ with values in a finite abelian group $G$, such that the associated graph $X(G,\alpha)$ is connected, we have that
    \[
    \lvert G \rvert \cdot \kappa(X(G,\alpha)) = \kappa(X) \cdot \prod_{\psi \neq \psi_0} \mathcal{I}_\alpha(\psi(1))
    \]
    where $\mathcal{I}_\alpha(\psi(1)) := \psi(\mathcal{I}_\alpha) \in \mathbb{C}$ is obtained by applying to $\mathcal{I}_\alpha$ the natural linear extension of the character $\psi$ to the group ring $\mathbb{Z}[t^G]$.
\end{proposition}
\begin{proof}
    First of all, observe that $\mathcal{I}_\alpha(\psi(1)) = \det(D_X - \widetilde{A}_\psi)$, where we define
    \[
        \widetilde{A}_{\psi} := \left(\sum_{\substack{e \in \mathbf{E}_{X} \\ {\rm inc}(e) = (v_{i},v_{j})}} \psi(\alpha(e)) \right) \in \mathbb{C}^{g \times g}
    \]
    for every character $\psi \in G^\vee$. In particular, one can prove that $\widetilde{A}_\psi = A_\psi$, as explained in \cite[Corollary~5.3]{McGown/Vallieres:2022a}.
    Therefore, we see from the definition of the polynomial $h_{Y/X}(u,\psi)$, which was given in \eqref{eq:h_polynomial}, that $\mathcal{I}_\alpha(\psi(1)) = h_{Y/X}(1,\psi)$ for every character $\psi \in G^\vee$. To conclude the proof, it is just sufficient to substitute this equality in the explicit expression
    \[
    |G| \cdot \kappa(Y) = \kappa(X) \prod_{\psi \neq \psi_{0}}h_{Y/X}(1,\psi)
    \]
    which was recalled in \eqref{cnf}.
\end{proof}

\subsubsection*{$\mathbb{Z}$-towers of graphs and Pierce-Lehmer sequences}

From now on, we will specialize to the case of $\mathbb{Z}$-towers of graphs.
More precisely, we will consider a finite connected graph $X$ such that $\chi(X) \neq 0$, endowed with a voltage assignment $\alpha \colon \mathbf{E}_X \to \mathbb{Z}$ with values in the additive group of the integers, such that the derived graph $X_\infty := X(\mathbb{Z},\alpha)$ is connected, which is equivalent to say that there exists a vertex $v_0 \in V_X$ such that the monodromy representation $\rho_{\alpha,v_0} \colon \pi_1(X,v_0) \to \mathbb{Z}$ is surjective, as explained in \cref{connectedness}.
In this case, we have a natural isomorphism $\mathbb{Z}[t^{\mathbb{Z}}] \cong \mathbb{Z}[t^{\pm 1}]$. 
Therefore, we see from \eqref{f_eq} that the Ihara polynomial can be written as
\[
\mathcal{I}_{\alpha}(t) = c_0 + c_1 (t + t^{-1}) + \dots + c_b (t^b + t^{-b})
\]
for some $c_0,\dots,c_b \in \mathbb{Z}$ such that $c_b \neq 0$. Clearing denominators, we can define
\[
I_\alpha(t) := t^b \mathcal{I}_\alpha(t) \in \mathbb{Z}[t],
\]
which is a polynomial of degree $2 b$ such that $I_\alpha(0) \neq 0$.
Finally, we define $e := \mathrm{ord}_{t = 1}(I_\alpha)$, and we write
\begin{equation} \label{eq:J_alpha}
    I_\alpha(t) = (t-1)^e J_\alpha(t)
\end{equation}
for some polynomial $J_\alpha \in \mathbb{Z}[t]$ such that $J_\alpha(0) \cdot J_\alpha(1) \neq 0$.

Now, one can associate to the voltage assignment $\alpha \colon \mathbf{E}_X \to \mathbb{Z}$ a system of finite graphs \[X_n := X(\mathbb{Z}/n \mathbb{Z},\pi_n \circ \alpha),\] where $\pi_n \colon \mathbb{Z} \twoheadrightarrow \mathbb{Z}/n \mathbb{Z}$ is the natural quotient map. 
In particular, each of these graphs will be connected, because $X_\infty$ is assumed to be connected, and the maps $\pi_n$ are surjective.
Moreover, the number of spanning trees of each graph $X_n$ can be computed in terms of a Pierce-Lehmer sequence associated to the polynomial $J_\alpha$, as the following result shows.

\begin{theorem} \label{thm:kN_Pierce_Lehmer}
    Let $X$ be a finite connected graph such that $\chi(X) \neq 0$, and fix a voltage assignment $\alpha \colon \mathbf{E}_X \to \mathbb{Z}$ such that the graph $X_\infty := X(\mathbb{Z},\alpha)$ is connected. Moreover, for every $n \in \mathbb{N}$ we let $X_n := X(\mathbb{Z}/n \mathbb{Z},\pi_n \circ \alpha)$, where $\pi_n \colon \mathbb{Z} \twoheadrightarrow \mathbb{Z}/n \mathbb{Z}$ is the natural quotient map. Then, the number of spanning trees of $X_n$ can be computed as
    \begin{equation} \label{eq:kN_Delta}
    \kappa(X_{n}) = (-1)^{b(n-1)} \cdot \kappa(X) \cdot n^{e-1} \cdot \frac{\Delta_{n}(J_{\alpha})}{\Delta_{1}(J_{\alpha})}
    \end{equation}
    where the integers $b := -\ord_{t = 0}(\mathcal{I}_\alpha) \geq 0$ and $e := \ord_{t = 1}(\mathcal{I}_\alpha) \geq 1$ are defined in terms of the Ihara polynomial $\mathcal{I}_\alpha \in \mathbb{Z}[t^{\pm 1}]$, whose definition was recalled in \eqref{eq:Ihara_polynomial}.
    Moreover, $\{ \Delta_n(J_\alpha) \}_{n \in \mathbb{N}}$ is the Pierce-Lehmer sequence, defined as in \eqref{eq:Pierce_Lehmer}, which is associated to the polynomial \[J_\alpha(t) := t^b \cdot (t-1)^{-e} \cdot \mathcal{I}_\alpha(t) \in \mathbb{Z}[t].\]
\end{theorem}
\begin{proof}
    First of all, observe that $e \geq 1$ because $\mathcal{I}_\alpha(1) = \det(D_X - A_X)$, where $D_X$ and $A_X$ are respectively the degree and adjacency matrices associated to $X$, whose difference $D_X - A_X$ is singular, as explained in \cite[Proposition~2.8]{Corry_Perkinson_2018}.
    Now, applying \cref{prop:Ihara_spanning_general} to the Galois cover $X_n/X$, we see that
    \begin{equation} \label{eq:spanning_trees_Zn}
        n \cdot \kappa(X_n) = \kappa(X) \cdot \prod_{\psi \neq \psi_0} \mathcal{I}_{\alpha_n}(\psi(1))
    \end{equation}
    where $\alpha_n := \pi_n \circ \alpha$ for every $n \in \mathbb{N}$.
    Moreover, it is easy to see using \eqref{eq:functorial_Ihara} that
    \begin{equation} \label{eq:Ihara_roots_of_unity}
        \prod_{\psi \neq \psi_0} \mathcal{I}_{\alpha_n}(\psi(1)) = \prod_{\psi \neq \psi_0} \mathcal{I}_\alpha(\psi(1)) = \prod_{\zeta \in W_n^\ast} \mathcal{I}_\alpha(\zeta),
    \end{equation}
    where $W_n^\ast := W_n \setminus \{1\}$ denotes the set of non-trivial roots of unity whose order divides $n$.  Now, let us observe that
    \begin{equation} \label{eq:resultant_roots_of_unity}
        \begin{aligned}
            \prod_{\zeta \in W_n^\ast} \mathcal{I}_\alpha(\zeta) &= \prod_{\zeta \in W_n^\ast} (\zeta^{-b} \cdot I_\alpha(\zeta)) = (-1)^{b (n-1)} \prod_{\zeta \in W_n^\ast} I_\alpha(\zeta) \\ &= (-1)^{b (n-1)} \mathrm{Res}\left( I_\alpha(t), \frac{t^n-1}{t-1} \right)
        \end{aligned}
    \end{equation}
    as follows from the definition of resultant recalled in \cref{sec:resultant}. 
    Thus, we see that
    \begin{equation} \label{resultant}
    n \cdot \kappa(X_{n}) = (-1)^{b(n-1)} \cdot \kappa(X) \cdot  {\rm Res} \left(I_{\alpha}(t),\frac{t^{n} - 1}{t-1} \right) 
    \end{equation}
    by combining \eqref{eq:spanning_trees_Zn} with \eqref{eq:Ihara_roots_of_unity} and \eqref{eq:resultant_roots_of_unity}.
    To conclude, it suffices to observe that
    \[
        \begin{aligned}
            {\rm Res} \left(I_{\alpha}(t),\frac{t^{n} - 1}{t-1} \right) &= {\rm Res} \left(t-1,\frac{t^{n} - 1}{t-1} \right)^e \cdot {\rm Res} \left(J_{\alpha}(t),\frac{t^{n} - 1}{t-1} \right) \\ &= n^e \cdot {\rm Res} \left(J_{\alpha}(t),\frac{t^{n} - 1}{t-1} \right) = n^e \cdot \frac{\Delta_n(J_\alpha)}{\Delta_1(J_\alpha)},
        \end{aligned}
    \]
    thanks to the multiplicative property of resultants.
\end{proof}

\begin{remark}
Formulas such as \eqref{resultant} appear also in the theory of curves over finite fields and in knot theory. 
Indeed:
\begin{itemize}
    \item if $C$ is a non-singular, geometrically irreducible projective curve over a finite field $\mathbb{F}_{q}$ with at least one rational point over $\mathbb{F}_{q}$, and $J$ is its Jacobian variety, then \cite[Corollary,~Page 110]{Rosen:2002} implies that
\[\#J(\mathbb{F}_{q^{n}}) = |{\rm Res}(P_{C}(t),t^{n}-1 )|, \]
where $P_{C}(t)$ is the Weil polynomial of $C$, defined as the reverse of the $L$-polynomial $L_{C}(t)$;
\item if $K \subseteq S^3$ is a knot, and $M_n$ is a Galois cover of $M := S^3 \setminus K$, with Galois group $\mathbb{Z}/n \mathbb{Z}$, Fox's formula (see \cite{Weber:1979}) implies that
\[\#H_{1}(X_{n},\mathbb{Z})_\text{tors} = |{\rm Res}(A_{K}(t), t^{n}-1)|,\]
where $A_{K}(t)$ is the Alexander polynomial associated to the knot $K$.
\end{itemize}
%Therefore, the Ihara polynomial associated to a voltage assignment is analogous to the Weil polynomial associated to towers of curves over finite fields \DV{I get confused with this analogy.  In some ways, the Weil polynomial should be analogous to $h_{X}$, I guess...}, and to the Alexander polynomial associated to a knot complement.
\end{remark}

\begin{remark} \label{rmk:p_l_adic_graphs}
    Let us note that, in the setting of \cref{thm:kN_Pierce_Lehmer}, the explicit formula \eqref{eq:kN_Delta} implies that $J_\alpha$ does not vanish at roots of unity. Indeed, if this was the case we would have $\kappa(X_n) = 0$ for some $n \geq 2$, which is absurd because $X_n$ is not empty.  
    Therefore, combining \cref{thm:kN_Pierce_Lehmer} with \cref{cor:PL_Iwasawa,cor:PL_Washington}, one can recover the formulas \eqref{eq:Iwasawa} and \eqref{eq:Washington} for the $\mathbb{Z}_\ell$-towers $\{X_{\ell^n}\}_{n = 0}^{+\infty}$ induced from a $\mathbb{Z}$-tower $\{X_n\}_{n = 1}^{+\infty}$.
    More generally, combining \cref{thm:kN_Pierce_Lehmer} with \cref{cor:PL_p_adic_valuation} will allow us to prove \cref{thm:main_thm_intro}, as we will explain in \cref{p_adic_val}.
\end{remark}
\begin{comment}
Let us note that, combining \cref{thm:kN_Pierce_Lehmer} with \cref{cor:PL_Iwasawa,cor:PL_Washington}, one can recover the formulas \eqref{eq:Iwasawa} and \eqref{eq:Washington} for the $\mathbb{Z}_\ell$-towers $\{X_{\ell^n}\}_{n = 0}^{+\infty}$    induced from a $\mathbb{Z}$-tower $\{X_n\}_{n = 1}^{+\infty}$.
\end{comment}

\begin{remark}
    Since the polynomial $J_{\alpha}$ is a reciprocal polynomial, it is known that the quantity $\lvert \Delta_{n}(J_{\alpha})/\Delta_{1}(J_{\alpha}) \rvert$ is a square when $n$ is odd and $J_{\alpha}$ is a monic polynomial, as explained for instance in \cite[Section~2]{EEW:2000}. 
    It follows that if $p$ and $\ell$ are two distinct rational primes with $\ell$ odd and $J_{\alpha}$ is monic, then
$${\rm ord}_{p}(\kappa(X_{\ell^{k}})) = {\rm ord}_{p}(\kappa(X)) + {\rm ord}_{p}\left(\Delta_{\ell^{k}}(J_{\alpha})/\Delta_{1}(J_{\alpha}) \right),  $$
and the parity of the number ${\rm ord}_{p}(\kappa(X_{\ell^{k}}))$, for all $k \ge 1$, depends only on the parity of ${\rm ord}_{p}(\kappa(X))$.  
This remark explains the parity of the $p$-adic valuation of the number of spanning trees in various $\mathbb{Z}_{\ell}$-towers appearing in \cite{Vallieres:2021, McGown/Vallieres:2022, McGown/Vallieres:2022a} and \cite{Lei/Vallieres:2023}, since in each case the tower in question is constructed from a voltage assignment $\alpha \colon \mathbf{E}_X \to \mathbb{Z}_\ell$ such that $\alpha(\mathbf{E}_X) \subseteq \mathbb{Z}$.  
We point out as well that the formula \eqref{eq:kN_Delta} is compatible with various results in the literature, such as \cite[Theorem 5.5]{Kwon/Mednykh:2017}, \cite[Theorem 5.5]{Mednykh:2018}, and \cite[Theorem 3]{Mednykh:2019}. 
\end{remark}

\subsection{An explicit example}
\label{sec:explicit_example_bouquet}

Let us revisit \cite[Example~2]{Vallieres:2021} using the results proven in the present paper. 
Consider the bouquet graph $X = B_{2}$ on two loops and pick an orientation $S = \{s_{1},s_{2} \}$.  Consider the function $\alpha:S \rightarrow \mathbb{Z}$ given by $\alpha(s_{1}) = 3$ and $\alpha(s_{2}) = 5$.  Note that $\alpha(s_{1}^{2} \cdot \bar{s}_{2})=1$, and thus \cref{connectedness} implies that $X(\mathbb{Z},\alpha)$ is connected. Therefore, so are all the finite graphs \[X_{n} := X(\mathbb{Z}/n \mathbb{Z},\pi_n \circ \alpha),\] where $\pi_n \colon \mathbb{Z} \twoheadrightarrow \mathbb{Z}/n \mathbb{Z}$ denotes the canonical projection map. 
The infinite graph $X(\mathbb{Z},\alpha)$ is a connected $4$-regular graph which is not a tree, but it is a quotient of the infinite $4$-regular tree.  All the finite graphs $X_{n}$ are finite quotients of $X(\mathbb{Z},\alpha)$. 
Moreover, one can draw each of these graphs $X_n$, using their definition, and we did so for $n \in \{ 1,\ldots,10,25,27 \}$ in the following figure, where we also drew a line $X_n \to X_m$ whenever $m \mid n$:
\input{old_circulant}

Note in particular that the $\mathbb{Z}_2$-tower considered in \cite[Example~2]{Vallieres:2021} corresponds to the leftmost column of the previous figure.
For the $\mathbb{Z}$-tower considered in the present example, the Ihara polynomial is given by
\begin{equation}
\begin{aligned}
\mathcal{I}_{\alpha}(t) = 4 - (t^3 + t^{-3}) - (t^5 + t^{-5}) = t^{-5} \cdot (t-1)^{2} \cdot J_{\alpha}(t), 
\end{aligned}
\end{equation}
with $J_{\alpha}(t) = -(t^{8} + 2t^{7} + 4t^{6} + 6t^{5}+8t^{4} + 6t^{3} + 4t^{2} + 2t + 1)$.  Using \textsc{SageMath} \cite{SAGE} 
, we computed the number of spanning trees $\kappa(X_{n})$, the resultants $\mathrm{Res}(I_\alpha(t),\frac{t^n-1}{t-1})$ and the values of the Pierce-Lehmer sequence $\Delta_n(J_\alpha)$ for each $n \in \{ 1,\ldots,10\}$. Doing so, we obtained the values which are tabulated in the following table, which shows in particular that the relationship between these invariants is the one predicted by \eqref{eq:kN_Delta}:
\begin{scriptsize}
\begin{center}
    \vspace{1ex}
    \begin{tabular}{c*{10}{c}}
    \toprule
        $n$ & 1 & 2 & 3 & 4 & 5 & 6 & 7 & 8 & 9 & 10 \\
    \midrule
        $\kappa(X_n)$ & 1 & 4 & 3 & 32 & 5 & 300 & 1183 & 1024 & 12321 & 16820 \\
    \midrule
        $\mathrm{Res}(I_\alpha(t),\frac{t^n-1}{t-1})$ & 1 & -8 & 9 & -128 & 25 & -1800 & 8281 & -8192 & 110889 & -168200 \\
    \midrule
        $\Delta_n(J_\alpha)$ & -34 & 68 & -34 & 272 & -34 & 1700 & -5746 & 4352 & -46546 & 57188 \\
    \bottomrule
    \end{tabular}
\end{center}
\end{scriptsize}

\subsection{Asymptotics for the number of spanning trees}
\label{sec:asymptotics}

The aim of this sub-section is to obtain some asymptotic results for the number of spanning trees in a $\mathbb{Z}$-tower of graphs, using the relation between the number of spanning trees and Pierce-Lehmer sequences, provided by \cref{thm:kN_Pierce_Lehmer}, in combination with the asymptotic results for Pierce-Lehmer sequences, which we explored in \cref{sec:Pierce_Lehmer}.

\subsubsection*{Archimedean asymptotics}

We will start from the Archimedean asymptotics of the number of spanning trees, which are provided by the following corollary of \cref{thm:kN_Pierce_Lehmer}.

\begin{corollary} \label{cor:archimedean_asypmtotics}
Let $X$ be a finite connected graph such that $\chi(X) \neq 0$, and fix a voltage assignment $\alpha \colon \mathbf{E}_X \to \mathbb{Z}$ such that $X(\mathbb{Z},\alpha)$ is connected. 
Then, if the polynomial $J_{\alpha}$ defined by \eqref{eq:J_alpha} does not have any root on the unit circle of $\mathbb{C}$, one has
\[
\kappa(X_{n}) \sim n^{e-1} \frac{\kappa(X)}{|\Delta_{1}(J_{\alpha})|} M_{\infty}(I_{\alpha})^{n}
\]
as $n \to +\infty$.
\end{corollary}
\begin{proof}
This follows directly by combining \eqref{eq:kN_Delta} with \eqref{euclidean_growth_away_uc}.
\end{proof}

\begin{remark}
    Let us note that given an Ihara polynomial $\mathcal{I}_\alpha$ associated to some voltage assignment $\alpha$, either $M_\infty(\mathcal{I}_\alpha) = 1$ or $M_\infty(\mathcal{I}_\alpha) \geq 2$, as was proved in \cite[Proposition~12.7]{Silver_Williams_2021}.
\end{remark} 

\begin{remark} \label{rmk:Weyl}
    It is reasonable to ask what happens when the Ihara polynomial $\mathcal{I}_\alpha$ has some roots on the Archimedean unit circle. 
    In this case, it is easy to see that these roots will prevent one from getting a precise asymptotic for the growth of $\kappa(X_n)$.
    To see this, fix some $\alpha = e^{2 \pi i \theta} \in \mathbb{C}$ with $\theta \in \mathbb{R} \setminus \mathbb{Q}$. Then, the sequence \[\lvert \alpha^n - 1 \rvert_\infty^2 = 2 (1 - \cos(2 \pi n \theta))\] is distributed on the interval $[0,4]$ according to the probability density function $\frac{2}{\pi \sqrt{4 x-x^2}}$, thanks to Weyl's equidistribution theorem \cite[Chapter 1, Example 2.1]{Kuipers_Niederreiter_1974}, and to the explicit computation of the probability density function of the random variable $2 (1-\cos(2 \pi U))$, where $U$ is a random variable which is uniformly distributed in the interval $[0,1]$, which follows from the general transformation formula for probability density functions (see \cite[Theorem~3.8.4]{DeGroot_Schervish_2012}).
    Therefore, we see that any asymptotic expansion for the growth of $\kappa(X_n)$ would have to feature some oscillating term, which takes into account this equidistribution phenomenon. 
\end{remark}
\begin{comment}
It is reasonable to ask what happens when the Ihara polynomial $\mathcal{I}_\alpha$ has some roots on the Archimedean unit circle. 
    In this case, it is easy to see that these roots will prevent one from getting a precise asymptotic for the growth of $\kappa(X_n)$.
    To see this, fix some $\alpha = e^{i \theta} \in \mathbb{C}$ with $\theta \in \mathbb{R} \setminus \mathbb{Q}$. Then, the sequence \[\lvert \alpha^n - 1 \rvert_\infty^2 = 2 (1 - \cos(n \theta))\] is distributed on the interval $[0,4]$ according to the probability density function $\frac{2}{\pi \sqrt{4 x-x^2}}$, thanks to Weyl's equidistribution theorem.
    Therefore, we see that any asymptotic expansion for the growth of $\kappa(X_n)$ would have to feature some oscillating term, which takes into account this equidistribution phenomenon. 
\end{comment}

The previous \cref{cor:archimedean_asypmtotics} allows us to recover the asymptotics for the number of spanning trees of two particular examples of $\mathbb{Z}$-towers, which were thoroughly studied in \cite{Mednykh:2018} and \cite{Mednykh:2019}.

\begin{example} \label{ex:Mednykh_1}
Consider the following orientation $S = \{s_1,s_2,s_3\}$ on the dumbbell graph:

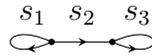
\begin{center}

\begin{tikzpicture}

%vertices
\draw[fill=black] (0,0) circle (1pt);
\draw[fill=black] (0.8,0) circle (1pt);

%edges
\draw (0,0) edge [decoration={markings, mark= at position 0.58 with {\arrow[xscale=1]{stealth}}},preaction={decorate}]  (0.8,0) node[above] at (0.4,0.08) {$s_{2}$};
\draw (0,0) edge [decoration={markings, mark= at position 0.28 with {\arrow{stealth}}},preaction = {decorate}, loop left, in = 155, out = 205,min distance=8mm] (0,0) node[above] at (-0.25,0.08) {$s_{1}$};
\draw (0.8,0) edge [decoration={markings, mark= at position 0.28 with {\arrow{stealth}}},preaction = {decorate}, loop right, in = 25, out = 335,min distance=8mm] (0.8,0) node[above] at (1.13,0.08) {$s_{3}$};
\end{tikzpicture}   
\end{center}
and fix a function $\alpha:S \rightarrow \mathbb{Z}$ such that $\alpha(s_2) = 0$.
This defines a voltage assignment on the dumbbell graph $X$, with values in $\mathbb{Z}$.
Moreover, the derived covers $X_n := X(\mathbb{Z}/n \mathbb{Z},\pi_n \circ \alpha)$, where $\pi_n \colon \mathbb{Z} \twoheadrightarrow \mathbb{Z}/n \mathbb{Z}$ denotes the canonical projection, are given by the $I$-graphs $I(n,k,l)$, where $k := \alpha(s_1)$ and $l := \alpha(s_3)$. In particular, if $k = 1$ one gets the family of generalized Petersen graphs $\mathrm{GP}(n,l)$.

Now, one sees from \cref{connectedness} that the graph $X(\mathbb{Z},\alpha)$ is connected if and only if $(k,l) = 1$, which we will assume for the rest of this example.
Then, we can compute the Ihara polynomial associated to the voltage assignment $\alpha$, which is given by
\[
\mathcal{I}_\alpha(t) = (3 - t^{k} - t^{-k})(3 - t^{l} - t^{-l}) - 1 = t^{-(k+l)} \cdot I_{\alpha}(t)
\]
where $I_{\alpha}(t) =  (3 t^k - t^{2 k} - 1) \cdot (3 t^l - t^{2 l} - 1) - t^{k+l}$.
Since $I_\alpha(1) = I'_\alpha(1) = 0$, while $I''_\alpha(1) \neq 0$, we see that $e := \ord_{t=1}(\mathcal{I}_\alpha) = 2$, which allows us to write 
\[
I_{\alpha}(t) = (t-1)^{2} \cdot J_{\alpha}(t)
\]
for some $J_\alpha \in \mathbb{Z}[t]$ such that $J_\alpha(1) \neq 0$.
A simple calculation shows that $J_\alpha$ has no roots on the unit circle, as explained for instance in \cite[Lemma~5.2]{Mednykh:2018}.
Moreover, it is easy to see that
\[
|\Delta_1(J_{\alpha})| = |J_{\alpha}(1)| = |I_{\alpha}^{''}(1)|/2 = k^2 + l^2,
\]
and that $\kappa(X) = 1$.
Therefore, \cref{cor:archimedean_asypmtotics} shows that
    \[
        \kappa(I(n,k,l)) \sim \frac{n}{k^2 + l^2} \cdot M_\infty(I_{\alpha})^n
    \]
    as $n \to \infty$, which is precisely \cite[Theorem 6.1]{Mednykh:2018}.    
\end{example}

\begin{example} \label{ex:Mednykh_2}
Consider the graph $X$ consisting of a bouquet with $k$ loops for some $k \in \mathbb{N}$, and take an orientation $S = \{s_{1},\ldots,s_{k} \}$ of $X$. 
Moreover, fix any function $\alpha:S \rightarrow \mathbb{Z}$ such that
$$1 \leq \alpha(s_1) < \alpha(s_2) < \ldots < \alpha(s_k), $$
and write $a_i := \alpha(s_i)$ for every $i \in \{1,\dots,k\}$.
Then, for $n$ large enough, the derived graph $X_n := X(\mathbb{Z}/n\mathbb{Z},\pi_n \circ \alpha)$ is the circulant graph $C_{n}(a_{1},\ldots,a_{k})$. 
Note in particular that the example described in \cref{sec:explicit_example_bouquet} belongs to this more general family.

Once again, it is easy to see by \cref{connectedness} that the graph $X(\mathbb{Z},\alpha)$ is connected if and only if $(a_{1},\ldots,a_{k}) = 1$, which we will assume for the rest of this example. Then, we can compute the Ihara polynomial associated to the voltage assignment $\alpha$, and we obtain 
\begin{equation*}
\mathcal{I}_{\alpha}(t) = 2k - \sum_{i=1}^{k}(t^{a_{i}} + t^{-a_{i}}) = t^{-a_k} I_\alpha(t)
\end{equation*}
where $I_\alpha(t) := 2 k t^{a_k} - \sum_{i=1}^k (t^{a_k+a_i} + t^{a_k-a_i}) \in \mathbb{Z}[t]$.
In particular, we easily see that 
\[I_{\alpha}(1) = I_{\alpha}'(1) = 0 \neq I_{\alpha}''(1),\] 
which implies that $e := \ord_{t = 1}(\mathcal{I}_\alpha) = 2$, and that 
$I_{\alpha}(t) = (t-1)^{2} \cdot J_{\alpha}(t)$,
for some $J_{\alpha} \in \mathbb{Z}[t]$ satisfying $J_{\alpha}(1) \neq 0$.
Moreover, one can easily show that $J_\alpha$ does not have any root on the unit circle of $\mathbb{C}$, as explained in \cite[Lemma~2]{Mednykh:2019}.
Finally, we see that $\kappa(X) = 1$ and 
\[
\lvert \Delta_1(J_{\alpha}) \rvert = \lvert J_{\alpha}(1) \rvert = \lvert I_{\alpha}^{''}(1) \rvert/2 = \sum_{i=1}^{k}a_{i}^{2}
\]
which, thanks to \cref{cor:archimedean_asypmtotics}, implies that
\[   
\kappa(C_{n}(a_{1},\ldots,a_{k})) \sim \frac{n}{\sum_{i=1}^{k}a_{i}^{2}} \cdot M_\infty(I_{\alpha})^n
\]
as $n \to \infty$, which is \cite[Theorem 5]{Mednykh:2019} in the particular case when $d=1$.       
\end{example}

\subsubsection*{$p$-adic asymptotics}

Let us look at the asymptotics of the $p$-adic valuations of the number of spanning trees in a $\mathbb{Z}$-tower, for a fixed prime $p$. As in the Archimedean setting, we start from the case when the Ihara polynomial  $\mathcal{I}_\alpha$ does not have any non-trivial root lying on the unit circle of $\mathbb{C}_p$.

\begin{corollary} \label{cor:p-adic_asymptotics}
Let $X$ be a finite connected graph such that $\chi(X) \neq 0$, and fix a voltage assignment $\alpha \colon \mathbf{E}_X \to \mathbb{Z}$ such that $X(\mathbb{Z},\alpha)$ is connected.
Fix moreover a rational prime $p \in \mathbb{N}$.
Then, if the polynomial $J_{\alpha}$ defined by \eqref{eq:J_alpha} has no root on the unit circle of $\mathbb{C}_{p}$, we have that
$$|\kappa(X_{n})|_{p} = |n|_{p}^{e-1}  \frac{|\kappa(X)|_{p}}{|\Delta_{1}(J_{\alpha})|_{p}} M_{p}(I_{\alpha})^{n}  $$
for every $n \in \mathbb{N}$, where $e := \mathrm{ord}_{t = 1}(\mathcal{I}_\alpha)$ is the order of vanishing at $t = 1$ of the Ihara polynomial associated to $\alpha$. 
\end{corollary}
\begin{proof}
    This follows immediately by combining \eqref{eq:kN_Delta} with \eqref{eq:p-adic_Pierce_Lehmer}.
\end{proof}
\begin{remark}
    Let $f(t) = \sum_{j = 0}^d c_j t^j \in \mathbb{Z}[t]$ be any polynomial. Then, every root of $f$ lies in the unit circle of $\mathbb{C}_p$ whenever $p \nmid c_0 \cdot c_d$. Therefore, we see that, for every given $\mathbb{Z}$-tower of graphs, \cref{cor:p-adic_asymptotics} can be applied only for finitely many primes $p$.
\end{remark}

\subsection{Exact formulas for the \texorpdfstring{$p$}{p}-adic valuation of the number of spanning trees}
\label{sec:exact_formulas}

The previous remark prompts us to study the case when $J_\alpha$ has some roots on the unit circle of $\mathbb{C}_p$.
In this case, we can prove the following result (which is the more precise version of \cref{thm:main_thm_intro} from the Introduction), which gives a partial analogue of Iwasawa's theorem for $\mathbb{Z}$-towers.
 
\begin{theorem} \label{p_adic_val}
Let $X$ be a finite connected graph whose Euler characteristic $\chi(X)$ does not vanish, and let $\alpha \colon \mathbf{E}_X \to \mathbb{Z}$ be a voltage assignment such that $X(\mathbb{Z},\alpha)$ is connected. 
Let $\mathcal{I}_\alpha \in \mathbb{Z}[t]$ be the Ihara polynomial associated to $\alpha$, and set \[J_\alpha(t) := t^b (t-1)^{-e} \mathcal{I}_\alpha(t),\] where $b := -\ord_{t=0}(\mathcal{I}_\alpha)$, and $e := \ord_{t=1}(\mathcal{I}_\alpha)$.

Fix now a rational prime $p \in \mathbb{N}$, an algebraic closure $\overline{\mathbb{Q}}_p$ of the field of $p$-adic numbers, and let $\mathcal{O}_p$ be the ring of integers of $\overline{\mathbb{Q}}_p$. Using this notation, we can define the quantities
\[
\begin{aligned}
    \mu_p(X,\alpha) &:= -m_p(J_\alpha)/\log(p) \\
c_p(X,\alpha) &:= {\rm ord}_{p}(\kappa(X)) - {\rm ord}_{p}\left(\Delta_{1}(J_{\alpha}) \right)
\end{aligned}
\]
where $m_p(J_\alpha)$ denotes the logarithmic $p$-adic Mahler measure of $J_\alpha$, defined as in \eqref{eq:log_padic_m}.

Moreover, let $\beta_1,\dots,\beta_d \in \overline{\mathbb{Q}}_p$ be the $p$-adic roots of $J_\alpha$, counted with multiplicity. Then, for every $n \in \mathbb{N}$ we introduce the set 
\begin{equation} \label{eq:Spn}
    B_{p,n}(X,\alpha) := \{ \beta \in \overline{\mathbb{Q}}_p \colon J_\alpha(\beta) = 0, \ \lvert \beta \rvert_p = 1, \ \lvert \beta^n - 1 \rvert_p < 1 \},
\end{equation}
which can be used to define the quantity 
\[\lambda_{p,n}(X,\alpha) := \# \{ j \in \{1,\dots,d\} \colon \beta_j \in B_{p,n}(X,\alpha) \} + e - 1.\]

Finally, for every $\beta \in \mathcal{O}_p$ such that $\lvert \beta \rvert_p = 1$ we write
\begin{equation} \label{eq:rpn}
    s_p(\beta) := \min\{ s \in \mathbb{Z}_{\ge 0} \colon p^s (p-1) \ord_p(\beta - \tau_p(\pi_p(\beta))) > 1\}
\end{equation} 
and for every $n \in \mathbb{N}$ we set $r_{p,n}(\beta) := \min(\ord_p(n),s_p(\beta))$, where $\tau_p(\pi_p(\beta))$ denotes the Teichmüller lift of the reduction $\pi_p(\beta)$ of $\beta$ modulo the maximal ideal of $\mathcal{O}_p$.
This can be used to define the quantity
\[
\nu_{p,n}(X,\alpha) := 
    \sum_{\substack{j \in \{1,\dots,d\} \\ \beta_j \in B_{p,n}(X,\alpha)}} \left( \ord_p(\beta_j^{p^{r_{p,n}(\beta_j)}} - \tau_p(\pi_p(\beta_j))^{p^{r_{p,n}(\beta_j)}}) - r_{p,n}(\beta_j) \right).
\]
Then, we have
\[
    \ord_p(\kappa(X_n)) = \mu_p(X,\alpha) \cdot n + \lambda_{p,n}(X,\alpha) \cdot \ord_p(n) + \nu_{p,n}(X,\alpha) + c_p(X,\alpha)
\]
for every $n \in \mathbb{N}$.
\end{theorem}
\begin{proof}
From \cref{thm:kN_Pierce_Lehmer}, one has the identity
\begin{equation} \label{eq:graph_to_PL}
    {\rm ord}_{p}(\kappa(X_{n})) = {\rm ord}_{p}(\Delta_{n}(J_{\alpha})) + (e-1) \cdot {\rm ord}_{p}(n) + c_p(X,\alpha).
\end{equation}
Moreover, \cref{cor:PL_p_adic_valuation} implies that
\begin{equation} \label{eq:PL_p_adic_valuation}
    \ord_p(\Delta_n(J_\alpha)) = \mu_p(X,\alpha) \cdot n + A_{p,n}(X,\alpha) \cdot \ord_p(n) + \nu_{p,n}(X,\alpha),
\end{equation}
where $A_{p,n}(X,\alpha) := \# \{ j \in \{1,\dots,d\} \colon \beta_j \in B_{p,n}(X,\alpha) \}$,
because $B_{p,n}(X,\alpha) = B_{p,n}(J_\alpha)$ and $\nu_{p,n}(X,\alpha) = \nu_{p,n}(J_\alpha)$ by definition. Therefore, we can conclude by combining \eqref{eq:PL_p_adic_valuation} with \eqref{eq:graph_to_PL}.
\end{proof}

Using \cref{p_adic_val}, one can easily show that once we fixed the base graph $X$, the voltage assignment $\alpha \colon \mathbf{E}_X \to \mathbb{Z}$ and the prime number $p$, we can subdivide $\mathbb{N}$ in a finite number of sequences, given by imposing certain divisibility conditions.
Along each of these sequences, the invariant $\ord_p(\kappa(X_n))$ can be computed as a linear form in $n$ and $\ord_p(n)$, as we show more precisely in the following theorem.

\begin{theorem}
    Let $X$ be a finite connected graph whose Euler characteristic $\chi(X)$ does not vanish, and $\alpha \colon \mathbf{E}_X \to \mathbb{Z}$ be a voltage assignment such that for every $n \geq 1$ the finite graph \[X_n := X(\mathbb{Z}/n \mathbb{Z},\alpha_n)\] is connected (which can be checked using \cref{connectedness}). 
    Moreover, for every prime $p \in \mathbb{Z}$ we write 
    \[\mu_p(X,\alpha) := -m_p(\mathcal{I}_\alpha)/\log(p),\] where $m_p(\mathcal{I}_\alpha)$ denotes the logarithmic $p$-adic Mahler measure of the Ihara polynomial $\mathcal{I}_\alpha$. 
    Then, for every rational prime $p$ there exist a finite set $\mathcal{N}_p(X,\alpha) \subseteq \mathbb{N}$ of integers coprime to $p$, and an integer $R_p(X,\alpha) \geq 0$ such that for every $\mathfrak{n} \subseteq \mathcal{N}_p(X,\alpha)$ and every $r \in \{0,\dots,R_p(X,\alpha)\}$ there exist two integers $\lambda_p(X,\alpha,\mathfrak{n}) \geq 0$ and $\nu_p(X,\alpha,\mathfrak{n},r)$ such that
    \begin{equation} \label{eq:main_formula}
        \ord_p(\kappa(X_n)) = \mu_p(X,\alpha) \cdot n + \lambda_p(X,\alpha,\mathfrak{n}) \cdot \ord_p(n) + \nu_p(X,\alpha,\mathfrak{n},r)
    \end{equation}
    for every $n \in \mathcal{S}_p(X,\alpha,\mathfrak{n},r)$, where $\mathcal{S}_p(X,\alpha,\mathfrak{n},r)$ consists of those $n \in \mathbb{N}$ such that: 
    \begin{itemize}
        \item $N \mid n$ for each $N \in \mathfrak{n}$;
        \item $N' \nmid n$ for each $N' \in \mathcal{N}_p(X,\alpha) \setminus \mathfrak{n}$;
        \item $\ord_p(n) = r$, if $r < R_p(X,\alpha)$, or $\ord_p(n) \geq R_p(X,\alpha)$ if $r = R_p(X,\alpha)$.
    \end{itemize}
    Moreover, the finite set $\mathcal{N}_p(X,\alpha)$, the integer $R_p(X,\alpha)$ and the two invariants $\lambda_p(X,\alpha,\mathfrak{n})$ and $\nu_p(X,\alpha,\mathfrak{n},r)$ can be explicitly computed in terms of the polynomial $\mathcal{I}_\alpha$. 
\end{theorem}
\begin{proof}
Fix a finite connected graph $X$ and a voltage assignment $\alpha \colon \mathbf{E}_X \to \mathbb{Z}$. 
    Moreover, fix a rational prime $p$ and let $\beta_1,\dots,\beta_d$ denote the roots of $J_\alpha$, counted with multiplicity, which lie on the unit circle of $\mathbb{C}_p$, and let $N_1,\dots,N_d$ denote the multiplicative orders of $\pi_p(\beta_1),\dots,\pi_p(\beta_d)$ in $\overline{\mathbb{F}}_p^\times$.
    Then, we can take
    \[
        \begin{aligned}
            \mathcal{N}_p(X,\alpha) &:= \{N_1,\dots,N_d\} \\
            R_p(X,\alpha) &:= \max_{j=1,\dots,d} s_p(\beta_j),
        \end{aligned}
    \]
    where $s_p(\beta_j)$ is defined as in \eqref{eq:sp_beta}. Moreover, for every subset $\mathfrak{n} \subseteq \mathcal{N}_p(X,\alpha)$ we can take 
    \[
        \lambda_p(X,\alpha,\mathfrak{n}) := \# \mathfrak{n} + e - 1,
    \]
    as we will now show.

    First of all, suppose that $\mathfrak{n} = \emptyset$. In other words, take an integer $n \in \mathbb{N}$ which is not divisible by any of the integers $N_1,\dots,N_d$. Then, the set $B_{p,n}(X,\alpha)$ defined in \eqref{eq:Spn} is empty.
    Therefore, \cref{p_adic_val} shows that for every $r \in \{0,\dots,R_p(X,\alpha)\}$ we can take
    \[         \nu_p(X,\alpha,\emptyset,r) := c_p(X,\alpha)
    \]
    and \eqref{eq:main_formula} will hold true.

    On the other hand, let us suppose that $\mathfrak{n} \neq \emptyset$, and let $J \subseteq \{1,\dots,d\}$ be the unique non-empty subset such that $\mathfrak{n} = \{N_j \colon j \in J \}$. 
    Then, if we suppose in addition that $p \neq 2$ and $p \nmid \mathrm{Disc}(J_\alpha)$, we have that $B_{p,n}(X,\alpha) = \{ \beta_j \colon j \in J \}$.
    Moreover, the quantity $r_{p,n}(\beta)$, which was defined in \eqref{eq:rpn}, vanishes whenever $\beta \in B_{p,n}(X,\alpha)$, because $\ord_p(\beta - \tau_p(\pi_p(\beta))) \in \mathbb{N}$ and $p-1 > 1$ in this case.
    Therefore, \cref{p_adic_val} shows that for every prime $p \neq 2$ such that $p \nmid \mathrm{Disc}(J_\alpha)$, and every $r \in \{0,\dots,R_p(X,\alpha)\}$, we can take
    \[
        \nu_p(X,\alpha,\mathfrak{n},r) := c_p(X,\alpha) + \sum_{j \in J} \ord_p(\beta_j - \tau_p(\pi_p(\beta_j))),
    \]
    and \eqref{eq:main_formula} will hold true.
    On the other hand, if $p = 2$ and $\mathrm{Disc}(J_\alpha)$ is odd, we still have that $B_{p,n}(X,\alpha) = \{\beta_j \colon j \in J\}$ and $\ord_p(\beta - \tau_p(\pi_p(\beta))) \in \mathbb{N}$ for every $\beta \in B_{p,n}(X,\alpha)$.
    Therefore, we see that for every $\beta \in B_{p,n}(X,\alpha)$ we have $r_{p,n}(\beta) = 0$ if $2 \nmid n$, and $r_{p,n}(\beta) \leq 1$ otherwise, because $\mathrm{Disc}(J_\alpha)$ is assumed to be odd.
    In other words, if $r = 0$ we can take
    \begin{equation} \label{eq:2-adic-valuation}
        \nu_2(X,\alpha,\mathfrak{n},0) := c_2(X,\alpha) +
        \sum_{j \in J} \ord_2(\beta_j - \tau_2(\pi_2(\beta_j))),    \end{equation}
    while if $r \geq 1$ we can take
    \[
        \begin{aligned}
            \nu_2(X,\alpha,\mathfrak{n},r) &:= c_2(X,\alpha) \\ &+ \sum_{\substack{j \in J \\ \mathrm{ord}_2(\beta_j - \tau_2(\pi_2(\beta_j))) = 1}} (\ord_2(\beta_j^2 - \tau_2(\pi_2(\beta_j))^2) - 1) \\ &+
        \sum_{\substack{j \in J \\ \mathrm{ord}_2(\beta_j - \tau_2(\pi_2(\beta_j))) \geq 2}} \ord_2(\beta_j - \tau_2(\pi_2(\beta_j)))
        \end{aligned}
    \]
    and \cref{p_adic_val} will ensure that \eqref{eq:main_formula} holds true when $p = 2$.

    To conclude, let us assume that $p \mid \mathrm{Disc}(J_\alpha)$, and let again $\mathfrak{n} = \{N_j \colon j \in J\}$ for some non-empty $J \subseteq \{1,\dots,d\}$, so that $B_{p,n}(X,\alpha) = \{\beta_j \colon j \in J\}$.
    Then, if $r = R_p(X,\alpha)$, which implies that $\ord_p(n) \geq R_{p}(X,\alpha)$, we see that
    \[
        r_{p,n}(\beta) = s_p(\beta) := \min\{ s \in \mathbb{Z}_{\geq 0} \colon p^s (p-1) \ord_p(\beta - \tau_p(\pi_p(\beta))) > 1 \}
    \]
    for every $\beta \in B_{p,n}(X,\alpha)$.
    Therefore, \cref{p_adic_val} guarantees that if we take
    \[
        \nu_p(X,\alpha,\mathfrak{n},R_p(X,\alpha)) := c_p(X,\alpha) + \sum_{j \in J} \left( \ord_p\left( \beta_j^{p^{s_p(\beta_j)}} - \tau_p(\pi_p(\beta_j))^{p^{s_p(\beta_j)}} \right) - s_p(\beta_j) \right),
    \]
    the identity \eqref{eq:main_formula} will hold true.
    Finally, if we fix $r \in \{0,\dots,R_p(X,\alpha) - 1\}$, we can take $\nu_p(X,\alpha,\mathfrak{n},r)$ to be
    \[
         c_p(X,\alpha) + \sum_{j \in J} \left( \ord_p\left( \beta_j^{p^{\min(r,s_p(\beta_j))}} - \tau_p(\pi_p(\beta_j))^{p^{\min(r,s_p(\beta_j))}} \right) - \min(r,s_p(\beta_j)) \right)
    \]
    and \cref{p_adic_val} still guarantees that \eqref{eq:main_formula} holds true.
\end{proof}
\begin{remark}
    The previous proof shows that the $p$-adic valuation of the number of spanning trees is actually constant along many of the sequences $\mathcal{S}_p(X,\alpha,\mathfrak{n},r)$.
    More precisely, let $\lVert J_\alpha \rVert$ be the greatest common divisor of the coefficients of $J_\alpha$. Then, if $p$ is a prime such that 
    \[p \nmid \mathrm{Disc}(J_\alpha) \cdot \lVert J_\alpha \rVert,\] for every $\mathfrak{n} = \{N_j \colon j \in J\} \subseteq \mathcal{N}_p(X,\alpha)$ we have that
    \[
        \ord_p(\kappa(X_n)) = c_p(X,\alpha) + \sum_{j \in J} \ord_p(\beta_j - \tau_p(\pi_p(\beta_j)))
    \]
    whenever $n \in \mathcal{S}_p(X,\alpha,\mathfrak{n},0)$.
\end{remark}

\begin{remark}
    It is clear from the proof of \cref{thm:main_thm_intro} that the multiplicative orders $N_1,\dots,N_d$ play a crucial role in the understanding of the evolution of the $p$-adic valuation of the number of spanning trees along a $\mathbb{Z}$-tower.
    Therefore, it would be nice to know how these orders vary with the prime number $p$.
    This can be understood in terms of a far reaching generalization of Artin's primitive root conjecture, due to Lenstra \cite{Lenstra_1977}, which is known to hold under the assumption of the generalized Riemann hypothesis. 
\end{remark}

% \begin{figure}[b]
%     \centering
%     \input{fibonacci_diagram}
%     \caption{The start of the Fibonacci tower, considered in \cref{sec:fibonacci}.}
%     \label{fig:fibonacci_tower}
% \end{figure}

\subsection{The Fibonacci tower} \label{sec:fibonacci}
To conclude this paper, let us note how the formula provided by \cref{p_adic_val} generalizes a famous formula for the $p$-adic valuation of the Fibonacci numbers, due to Lengyel \cite{Lengyel:1995}.  More precisely, if in \cref{ex:Mednykh_2}, we take the bouquet on two loops $X$, with an orientation $S = \{s_1,s_2\}$, and we let $\alpha$ be the unique voltage assignment $\alpha:\mathbf{E}_{X} \rightarrow \mathbb{Z}$ such that $\alpha(s_1) = 1$ and $\alpha(s_2) = 2$, then we obtain the $\mathbb{Z}$-tower portrayed in the following figure: 
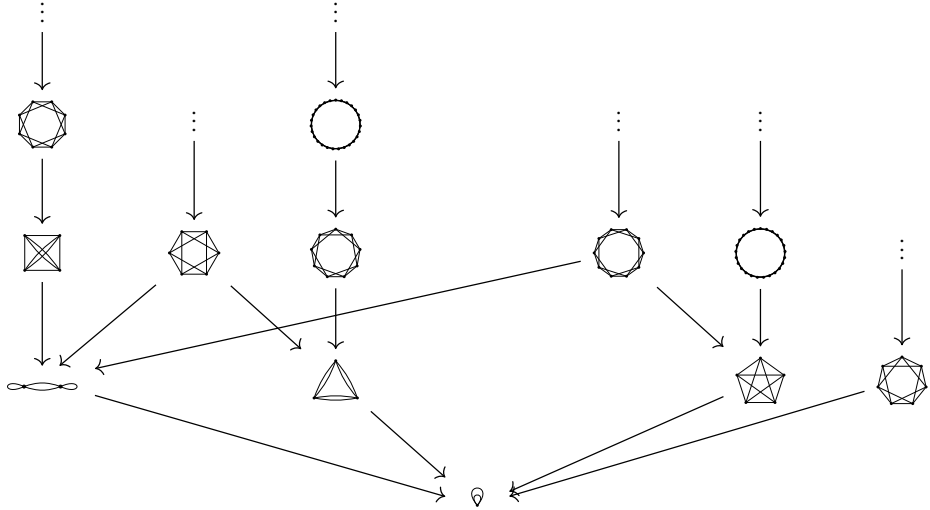
\input{fibonacci_diagram}

It turns out that the number of spanning trees of the finite layers of this tower is intimately related to the sequence of Fibonacci numbers. To show this, let us observe that $\kappa(X) = 1$ and 
\[\mathcal{I}_\alpha(t) := 4 - (t+t^{-1}) - (t^2+t^{-2}),\] 
which implies that $e = 2$ and $J_\alpha(t) = -(t^2 + 3 t + 1)$.
We denote by $\beta_1 = \frac{-3 - \sqrt{5}}{2}$ and $\beta_2 = \frac{-3+\sqrt{5}}{2}$ the roots of $J_\alpha$, and we observe that $\Delta_1(J_\alpha) = -5$. 

Then, \cref{thm:kN_Pierce_Lehmer} can be combined with a simple computation to show that 
\[
\kappa(X_{n}) = n \frac{\Delta_{n}(J_{\alpha})}{\Delta_{1}(J_{\alpha})} = n \frac{(-1)^{n}(\beta_{1}^{n} - 1)(\beta_{2}^{n} - 1)}{-5} = n F_{n}^{2},
\]
where $F_{n}$ is the $n$th Fibonacci number.  Therefore, \cref{p_adic_val} implies that
\[
    2 \cdot \ord_p(F_n) = \ord_p(\kappa(X_n)) - \ord_p(n) = \# B_{p,n}(X,\alpha) \cdot \ord_p(n) + \nu_{p,n}(X,\alpha)
\]
for every prime $p \neq 5$ and every $n \in \mathbb{N}$, because $m_p(J_\alpha) = 0$ for every prime $p \in \mathbb{N}$.

Now, let us note that the two roots $\beta_1$ and $\beta_2$ of the polynomial $J_\alpha$ are both reciprocal units, which implies that the set $B_{p,n}(X,\alpha)$ is either empty or consists of the two roots $\{\beta_1,\beta_2\}$.
The latter scenario occurs if and only if $n$ is a multiple of the multiplicative order $N_p$ of $\beta_1$ (and $\beta_2$) in $\overline{\mathbb{F}}_p^\times$.
Therefore, \cref{p_adic_val} shows that $N_p$ coincides with the so called rank of apparition of the prime $p$, \textit{i.e.} with the smallest index $n$ such that $p \mid F_n$.
Since there exists $k \in \mathbb{N}$ such that $N_p \mid p^k - 1$, we see that $\ord_p(N_p) = 0$, and thus that $\sum_{j=1}^2 \ord_p(\beta_j - \tau_p(\pi_p(\beta_j))) = 2 \cdot \ord_p(F_{N_p})$, provided one assumes moreover that $p \neq 2$. 
Hence, these considerations entail that
\begin{equation*}
    \ord_p(F_n) = \begin{cases}
    \ord_p(n) + \ord_p(F_{N_p}), \ \text{if} \ N_p \mid n, \\
    0, \ \text{otherwise},
\end{cases}
\end{equation*}
for every $p \neq 2,5$, as was proven in \cite[\S~3]{Lengyel:1995}.

To conclude this example, and this paper, let us see what happens when $p = 2$ and $p = 5$.  In the first case, when $p=2$, we can observe that $J_\alpha$ has no roots in $\mathbb{F}_2$, which implies necessarily that $N_2 = 3$. Moreover, in order to compute the Teichmüller representatives $\tau_2(\pi_2(\beta_j))$, for $j \in \{1,2\}$, we can work globally, and consider the number field $K = \mathbb{Q}(\sqrt{5},\zeta_3)$, where $\zeta_3$ is a primitive third root of unity.  A simple calculation shows that $2$ decomposes as a product of two primes in $K$ both with inertia degree $2$ and ramification degree $1$.  Using \textsc{SageMath} \cite{SAGE}, or by hand, one calculates that $(\beta_1 - \zeta_{3}) = \mathfrak{p}$, where $\mathfrak{p}$ is one of the two primes lying above $2$.  If we denote the other prime lying above $2$ by $\mathfrak{q}$, then we also have $(\beta_1 - \zeta_{3}^{2}) = \mathfrak{q}$.  Moreover, we have $(\beta_1^{2} - \zeta_{3}^{2}) = \mathfrak{p}^{3}$, and $(\beta_1^{2} - \zeta_{3}^{4}) = \mathfrak{q}^{3}$.  A similar calculation can be performed for $\beta_{2}$.  After embedding $K$ into $\overline{\mathbb{Q}}_{2}$ with any embedding, these calculations show that $N_{2} = 3$ and 
$$\ord_2(\beta_j - \tau_2(\pi_2(\beta_j))) = 1, \ord_2(\beta_j^2 - \tau_2(\pi_2(\beta_j^2))) = 3, $$
for every $j \in \{1,2\}$.  Finally, one can check easily that $\ord_2(F_3) = 1$. Thus, \eqref{eq:rpn} implies that
\[
\ord_2(F_n) = \begin{cases}
    0, \ & \text{if} \ n \equiv 1,2 \mod 3, \\
    1, \ & \text{if} \ n \equiv 3 \mod 6, \\
    \ord_2(n) + 2, \ & \text{if} \ n \equiv 0 \mod 6, 
\end{cases}
\]
which was proven in \cite[Lemma~2]{Lengyel:1995}.

Let us now suppose that $p = 5$. Since $-J_\alpha \equiv (t-1)^2 \ (\text{mod} \ 5)$, we see that the multiplicative order of $\beta_{1}$ and $\beta_{2}$ in $\overline{\mathbb{F}}_{5}^{\times}$ is $1$, which is not the rank of apparition of the prime $5$. 
Moreover, one has that $\ord_5(\beta_j - \tau_5(\pi_5(\beta_j))) = \frac{1}{2}$. Indeed, one sees that $\tau_5(\pi_5(\beta_j)) = \tau_5(1) = 1$. Hence 
\[
\ord_5(\beta_j - \tau_5(\pi_5(\beta_j))) = \mathrm{ord}_5(\sqrt{5}) = \frac{1}{2}\] for every $j \in \{1,2\}$, as we wanted to show. Finally, let us observe that $s_5(\beta_j) = 0$ for every $j \in \{1,2\}$. Combining this with the fact that $c_5(X,\alpha) = -1$, we see that
\[
    \ord_5(F_n) = \ord_5(n) + \frac{1}{4} + \frac{1}{4} - \frac{1}{2} = \ord_5(n)
\]
for every $n \in \mathbb{N}$, as was proven in \cite[Lemma~1]{Lengyel:1995}. 
This shows that \cref{p_adic_val} can be seen as a generalization of Lengyel's theorem to sequences that arise as the number of spanning trees in a $\mathbb{Z}$-cover of finite graphs.

\section*{Acknowledgments}

We thank Sören Kleine, Matilde Lalín, Katharina Müller and Antonio Lei for some comments on an earlier version of the present paper.
Moreover, the first named author thanks Francesco Campagna, Roberto Gualdi, Pieter Moree and Fabien Pazuki for useful discussions and remarks.  We would also like to thank the anonymous referee for thoroughly going through the paper and making many valuable comments, remarks, and suggestions.

\section*{Funding}
Both authors are grateful to the Max Planck Institute for Mathematics in Bonn for providing excellent working conditions, great hospitality and financial support.
Moreover, the first named author thanks the research projects ``Motivic homotopy, quadratic invariants and diagonal classes'' (ANR-21-CE40-0015) and IRN GANDA for their financial support.

\vspace{\baselineskip}
\noindent
\framebox[\textwidth]{
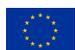
\begin{tabular*}{0.96\textwidth}{@{\extracolsep{\fill} }cp{0.84\textwidth}}
 % The EU emblem
\raisebox{-0.7\height}{%
    \begin{tikzpicture}[y=0.80pt, x=0.8pt, yscale=-1, inner sep=0pt, outer sep=0pt, 
    scale=0.12]
    \definecolor{c003399}{RGB}{0,51,153}
    \definecolor{cffcc00}{RGB}{255,204,0}
    \begin{scope}[shift={(0,-872.36218)}]
      \path[shift={(0,872.36218)},fill=c003399,nonzero rule] (0.0000,0.0000) rectangle (270.0000,180.0000);
      \foreach \myshift in 
           {(0,812.36218), (0,932.36218), 
    		(60.0,872.36218), (-60.0,872.36218), 
    		(30.0,820.36218), (-30.0,820.36218),
    		(30.0,924.36218), (-30.0,924.36218),
    		(-52.0,842.36218), (52.0,842.36218), 
    		(52.0,902.36218), (-52.0,902.36218)}
        \path[shift=\myshift,fill=cffcc00,nonzero rule] (135.0000,80.0000) -- (137.2453,86.9096) -- (144.5106,86.9098) -- (138.6330,91.1804) -- (140.8778,98.0902) -- (135.0000,93.8200) -- (129.1222,98.0902) -- (131.3670,91.1804) -- (125.4894,86.9098) -- (132.7547,86.9096) -- cycle;
    \end{scope}
    %\draw[very thin,dashed] (current bounding box.south west) rectangle               (current bounding box.north east);
    \end{tikzpicture}%
}
&
Riccardo Pengo received funding from the European Research Council (ERC) under the European Union’s Horizon 2020 research and innovation programme (grant agreement number 945714).
\end{tabular*}
}

%\bibliographystyle{alpha}
%\bibliography{references}
\AtNextBibliography{\footnotesize}
\emergencystretch=1em
%\appto{\bibsetup}{\sloppy}
\printbibliography
\end{document}

%% file: old_circulant.tex
\begin{equation*} 
\begin{tikzcd}[sep=1.8em, font=\small,scale cd=0.8,nodes in empty cells]
%& & & \tikzfig{prova} \arrow[dlll,dashed] \arrow[ddll,dashed] \arrow[dl,dashed] \arrow[dddd,dashed] \arrow[ddr,dashed] \arrow[ddrr,dashed] \arrow[dddrrr,dashed] & & & \\
\vdots \arrow[d] & & \vdots \arrow[d]& & & & \\
\begin{tikzpicture}[baseline={([yshift=-0.6ex] current bounding box.center)}]
% create the node
\node[draw=none,minimum size=1cm,regular polygon,regular polygon sides=8] (a) {};

% draw a black dot in each vertex
\foreach \x in {1,2,...,8}
  \fill (a.corner \x) circle[radius=0.7pt];
  
\foreach \y\z in {1/4,2/5,3/6,4/7,5/8,6/1,7/2,8/3}
  \path (a.corner \y) edge [bend left=10] (a.corner \z);
  
\foreach \y\z in {1/6,2/7,3/8,4/1,5/2,6/3,7/4,8/5}
  \path (a.corner \y) edge [bend left=10] (a.corner \z);
\end{tikzpicture} \arrow[d]& \vdots \arrow[d] & \begin{tikzpicture}[baseline={([yshift=-0.6ex] current bounding box.center)}]
% create the node
\node[draw=none,minimum size=1cm,regular polygon,regular polygon sides=27] (a) {};

% draw a black dot in each vertex
\foreach \x in {1,2,...,27}
  \fill (a.corner \x) circle[radius=0.7pt];
  
\foreach \y\z in {1/4,2/5,3/6,4/7,5/8,6/9,7/10,8/11,9/12,10/13,11/14,12/15,13/16,14/17,15/18,16/19,17/20,18/21,19/22,20/23,21/24,22/25,23/26,24/27,25/1,26/2,27/3}
  \path (a.corner \y) edge  (a.corner \z);
  
\foreach \y\z in {1/6,2/7,3/8,4/9,5/10,6/11,7/12,8/13,9/14,10/15,11/16,12/17,13/18,14/19,15/20,16/21,17/22,18/23,19/24,20/25,21/26,22/27,23/1,24/2,25/3,26/4,27/5}
  \path (a.corner \y) edge  (a.corner \z);

\end{tikzpicture} \arrow[d]& & \vdots \arrow[d] & \vdots \arrow[d] & \hspace{1cm} \\
\begin{tikzpicture}[baseline={([yshift=-0.6ex] current bounding box.center)}]
% create the node
\node[draw=none,minimum size=1cm,regular polygon,regular polygon sides=4] (a) {};

% draw a black dot in each vertex
\foreach \x in {1,2,...,4}
  \fill (a.corner \x) circle[radius=0.7pt];

\path (a.corner 1) edge [bend left=20] (a.corner 2);
\path (a.corner 1) edge [bend right=20] (a.corner 2);
\path (a.corner 2) edge [bend left=20] (a.corner 3);
\path (a.corner 2) edge [bend right=20] (a.corner 3);
\path (a.corner 3) edge [bend left=20] (a.corner 4);
\path (a.corner 3) edge [bend right=20] (a.corner 4);
\path (a.corner 4) edge [bend left=20] (a.corner 1);
\path (a.corner 4) edge [bend right=20] (a.corner 1);
\end{tikzpicture} \arrow[d]& \begin{tikzpicture}[baseline={([yshift=-0.6ex] current bounding box.center)}]
% create the node
\node[draw=none,minimum size=1cm,regular polygon,regular polygon sides=6] (a) {};

% draw a black dot in each vertex
\foreach \x in {1,2,...,6}
  \fill (a.corner \x) circle[radius=0.7pt];
  
\foreach \y\z in {1/4,2/5,3/6,4/1,5/2,6/3}
  \path (a.corner \y) edge [bend left=10] (a.corner \z);
  
\foreach \y\z in {1/6,2/1,3/2,4/3,5/4,6/5}
  \path (a.corner \y) edge   (a.corner \z);

\end{tikzpicture} \arrow[dl] \arrow[dr]& \begin{tikzpicture}[baseline={([yshift=-0.6ex] current bounding box.center)}]
% create the node
\node[draw=none,minimum size=1cm,regular polygon,regular polygon sides=9] (a) {};

% draw a black dot in each vertex
\foreach \x in {1,2,...,9}
  \fill (a.corner \x) circle[radius=0.7pt];
  
\foreach \y\z in {1/4,2/5,3/6,4/7,5/8,6/9,7/1,8/2,9/3}
  \path (a.corner \y) edge  (a.corner \z);
  
\foreach \y\z in {1/6,2/7,3/8,4/9,5/1,6/2,7/3,8/4,9/5}
  \path (a.corner \y) edge   (a.corner \z);

\end{tikzpicture} \arrow[d]& \hspace{1cm} & \begin{tikzpicture}[baseline={([yshift=-0.6ex] current bounding box.center)}]
% create the node
\node[draw=none,minimum size=1cm,regular polygon,regular polygon sides=10] (a) {};

% draw a black dot in each vertex
\foreach \x in {1,2,...,10}
  \fill (a.corner \x) circle[radius=0.7pt];
  
\foreach \y\z in {1/4,2/5,3/6,4/7,5/8,6/9,7/10,8/1,9/2,10/3}
  \path (a.corner \y) edge (a.corner \z);
  
\foreach \y\z in {1/6,2/7,3/8,4/9,5/10}
  \path (a.corner \y) edge [bend left=10] (a.corner \z);
  
\foreach \y\z in {6/1,7/2,8/3,9/4,10/5}
  \path (a.corner \y) edge [bend left=10] (a.corner \z);

\end{tikzpicture} \arrow[dllll] \arrow[dr]& \begin{tikzpicture}[baseline={([yshift=-0.6ex] current bounding box.center)}]
% create the node
\node[draw=none,minimum size=1cm,regular polygon,regular polygon sides=25] (a) {};

% draw a black dot in each vertex
\foreach \x in {1,2,...,25}
  \fill (a.corner \x) circle[radius=0.7pt];
  
\foreach \y\z in {1/4,2/5,3/6,4/7,5/8,6/9,7/10,8/11,9/12,10/13,11/14,12/15,13/16,14/17,15/18,16/19,17/20,18/21,19/22,20/23,21/24,22/25,23/1,24/2,25/3}
  \path (a.corner \y) edge  (a.corner \z);
  
\foreach \y\z in {1/6,2/7,3/8,4/9,5/10,6/11,7/12,8/13,9/14,10/15,11/16,12/17,13/18,14/19,15/20,16/21,17/22,18/23,19/24,20/25,21/1,22/2,23/3,24/4,25/5}
  \path (a.corner \y) edge  (a.corner \z);

\end{tikzpicture} \arrow[d] & \vdots \arrow[d] \\
\begin{tikzpicture}[baseline={([yshift=-0.6ex] current bounding box.center)}]
% create the node
\node[draw=none,minimum size=1cm,regular polygon,rotate = -45,regular polygon sides=4] (a) {};

% draw a black dot in each vertex
  \fill (a.corner 1) circle[radius=0.7pt];
  \fill (a.corner 3) circle[radius=0.7pt];
  
  \path (a.corner 1) edge [bend left=20] (a.corner 3);
  \path (a.corner 1) edge [bend left=60] (a.corner 3);
  \path (a.corner 1) edge [bend right=20] (a.corner 3);
  \path (a.corner 1) edge [bend right=60] (a.corner 3);

\end{tikzpicture} \arrow[drrr]& \hspace{1cm}& \begin{tikzpicture}[baseline={([yshift=-0.6ex] current bounding box.center)}]
% create the node
\node[draw=none,minimum size=1cm,regular polygon,regular polygon sides=3] (a) {};

% draw a black dot in each vertex
\foreach \x in {1,2,...,3}
  \fill (a.corner \x) circle[radius=0.7pt];
\draw (a.corner 1) to [in=50,out=130,distance = 0.3cm,loop] (a.corner 1);
\draw (a.corner 2) to [in=50,out=130,distance = 0.3cm,loop] (a.corner 2);
\draw (a.corner 3) to [in=50,out=130,distance = 0.3cm,loop] (a.corner 3);
  
\foreach \y\z in {1/3,2/1,3/2}
  \path (a.corner \y) edge (a.corner \z);
\end{tikzpicture}  \arrow[dr]& \hspace{1cm}& \hspace{1cm}&\begin{tikzpicture}[baseline={([yshift=-0.6ex] current bounding box.center)}]
% create the node
\node[draw=none,minimum size=1cm,regular polygon,regular polygon sides=5] (a) {};

% draw a black dot in each vertex
\foreach \x in {1,2,...,5}
  \fill (a.corner \x) circle[radius=0.7pt];
\draw (a.corner 1) to [in=50,out=130,distance = 0.3cm,loop] (a.corner 1);
\draw (a.corner 2) to [in=50,out=130,distance = 0.3cm,loop] (a.corner 2);
\draw (a.corner 3) to [in=50,out=130,distance = 0.3cm,loop] (a.corner 3);
\draw (a.corner 4) to [in=50,out=130,distance = 0.3cm,loop] (a.corner 4);
\draw (a.corner 5) to [in=50,out=130,distance = 0.3cm,loop] (a.corner 5);
  
\foreach \y\z in {1/4,2/5,3/1,4/2,5/3}
  \path (a.corner \y) edge (a.corner \z);

\end{tikzpicture} \arrow[dll] & \begin{tikzpicture}[baseline={([yshift=-0.6ex] current bounding box.center)}]
% create the node
\node[draw=none,minimum size=1cm,regular polygon,regular polygon sides=7] (a) {};

% draw a black dot in each vertex
\foreach \x in {1,2,...,7}
  \fill (a.corner \x) circle[radius=0.7pt];
  
\foreach \y\z in {1/4,2/5,3/6,4/7,5/1,6/2,7/3}
  \path (a.corner \y) edge  (a.corner \z);
  
\foreach \y\z in {1/6,2/7,3/1,4/2,5/3,6/4,7/5}
  \path (a.corner \y) edge   (a.corner \z);

\end{tikzpicture} \arrow[dlll] \\
\hspace{1cm}& \hspace{1cm}& \hspace{1cm}& \begin{tikzpicture}[baseline={([yshift=-1.7ex] current bounding box.center)}]
% create the node
\node[draw=none,minimum size=1cm,regular polygon,regular polygon sides=1] (a) {};
% draw a black dot in each vertex
\foreach \x in {1}
  \fill (a.corner \x) circle[radius=0.7pt];
\draw (a.corner 1) to [in=50,out=130,distance = 0.5cm,loop] (a.corner 1);
\draw (a.corner 1) to [in=50,out=130,distance = 0.3cm,loop] (a.corner 1);
\end{tikzpicture}&\hspace{1cm} &\hspace{1cm} & \hspace{1cm}
\end{tikzcd}
\end{equation*}

%% file: fibonacci_diagram.tex
\begin{equation*} 
\begin{tikzcd}[sep=1.8em, font=\small,scale cd=0.8,nodes in empty cells]
%& & & \tikzfig{prova} \arrow[dlll,dashed] \arrow[ddll,dashed] \arrow[dl,dashed] \arrow[dddd,dashed] \arrow[ddr,dashed] \arrow[ddrr,dashed] \arrow[dddrrr,dashed] & & & \\
\vdots \arrow[d] & & \vdots \arrow[d]& & & & \\
\begin{tikzpicture}[baseline={([yshift=-0.6ex] current bounding box.center)}]
% create the node
\node[draw=none,minimum size=1cm,regular polygon,regular polygon sides=8] (a) {};

% draw a black dot in each vertex
\foreach \x in {1,2,...,8}
  \fill (a.corner \x) circle[radius=0.7pt];
  
\foreach \y\z in {1/2,2/3,3/4,4/5,5/6,6/7,7/8,8/1}
  \path (a.corner \y) edge  (a.corner \z);
  
\foreach \y\z in {1/3,2/4,3/5,4/6,5/7,6/8,7/1,8/2}
  \path (a.corner \y) edge  (a.corner \z);
\end{tikzpicture} \arrow[d]& \vdots \arrow[d] & \begin{tikzpicture}[baseline={([yshift=-0.6ex] current bounding box.center)}]
% create the node
\node[draw=none,minimum size=1cm,regular polygon,regular polygon sides=27] (a) {};

% draw a black dot in each vertex
\foreach \x in {1,2,...,27}
  \fill (a.corner \x) circle[radius=0.7pt];
  
\foreach \y\z in {1/2,2/3,3/4,4/5,5/6,6/7,7/8,8/9,9/10,10/11,11/12,12/13,13/14,14/15,15/16,16/17,17/18,18/19,19/20,20/21,21/22,22/23,23/24,24/25,25/26,26/27,27/1}
  \path (a.corner \y) edge  (a.corner \z);
  
\foreach \y\z in {1/3,2/4,3/5,4/6,5/7,6/8,7/9,8/10,9/11,10/12,11/13,12/14,13/15,14/16,15/17,16/18,17/19,18/20,19/21,20/22,21/23,22/24,23/25,24/26,25/27,26/1,27/2}
  \path (a.corner \y) edge  (a.corner \z);

\end{tikzpicture} \arrow[d]& & \vdots \arrow[d] & \vdots \arrow[d] & \hspace{1cm} \\
\begin{tikzpicture}[baseline={([yshift=-0.6ex] current bounding box.center)}]
% create the node
\node[draw=none,minimum size=1cm,regular polygon,regular polygon sides=4] (a) {};

% draw a black dot in each vertex
\foreach \x in {1,2,...,4}
  \fill (a.corner \x) circle[radius=0.7pt];

\path (a.corner 1) edge (a.corner 2);
\path (a.corner 1) edge [bend right=10] (a.corner 3);
\path (a.corner 2) edge  (a.corner 3);
\path (a.corner 2) edge [bend right=10] (a.corner 4);
\path (a.corner 3) edge  (a.corner 4);
\path (a.corner 2) edge [bend left=10] (a.corner 4);
\path (a.corner 4) edge  (a.corner 1);
\path (a.corner 1) edge [bend left=10] (a.corner 3);
\end{tikzpicture} \arrow[d]& \begin{tikzpicture}[baseline={([yshift=-0.6ex] current bounding box.center)}]
% create the node
\node[draw=none,minimum size=1cm,regular polygon,regular polygon sides=6] (a) {};

% draw a black dot in each vertex
\foreach \x in {1,2,...,6}
  \fill (a.corner \x) circle[radius=0.7pt];
  
\foreach \y\z in {1/2,2/3,3/4,4/5,5/6,6/1}
  \path (a.corner \y) edge  (a.corner \z);
  
\foreach \y\z in {1/3,2/4,3/5,4/6,5/1,6/2}
  \path (a.corner \y) edge   (a.corner \z);

\end{tikzpicture} \arrow[dl] \arrow[dr]& \begin{tikzpicture}[baseline={([yshift=-0.6ex] current bounding box.center)}]
% create the node
\node[draw=none,minimum size=1cm,regular polygon,regular polygon sides=9] (a) {};

% draw a black dot in each vertex
\foreach \x in {1,2,...,9}
  \fill (a.corner \x) circle[radius=0.7pt];
  
\foreach \y\z in {1/2,2/3,3/4,4/5,5/6,6/7,7/8,8/9,9/1}
  \path (a.corner \y) edge  (a.corner \z);
  
\foreach \y\z in {1/3,2/4,3/5,4/6,5/7,6/8,7/9,8/1,9/2}
  \path (a.corner \y) edge   (a.corner \z);

\end{tikzpicture} \arrow[d]& \hspace{1cm} & \begin{tikzpicture}[baseline={([yshift=-0.6ex] current bounding box.center)}]
% create the node
\node[draw=none,minimum size=1cm,regular polygon,regular polygon sides=10] (a) {};

% draw a black dot in each vertex
\foreach \x in {1,2,...,10}
  \fill (a.corner \x) circle[radius=0.7pt];
  
\foreach \y\z in {1/2,2/3,3/4,4/5,5/6,6/7,7/8,8/9,9/10,10/1}
  \path (a.corner \y) edge (a.corner \z);
  
\foreach \y\z in {1/3,2/4,3/5,4/6,5/7}
  \path (a.corner \y) edge  (a.corner \z);
  
\foreach \y\z in {6/8,7/9,8/10,9/1,10/2}
  \path (a.corner \y) edge  (a.corner \z);

\end{tikzpicture} \arrow[dllll] \arrow[dr]& \begin{tikzpicture}[baseline={([yshift=-0.6ex] current bounding box.center)}]
% create the node
\node[draw=none,minimum size=1cm,regular polygon,regular polygon sides=25] (a) {};

% draw a black dot in each vertex
\foreach \x in {1,2,...,25}
  \fill (a.corner \x) circle[radius=0.7pt];
  
\foreach \y\z in {1/2,2/3,3/4,4/5,5/6,6/7,7/8,8/9,9/10,10/11,11/12,12/13,13/14,14/15,15/16,16/17,17/18,18/19,19/20,20/21,21/22,22/23,23/24,24/25,25/1}
  \path (a.corner \y) edge  (a.corner \z);
  
\foreach \y\z in {1/3,2/4,3/5,4/6,5/7,6/8,7/9,8/10,9/11,10/12,11/13,12/14,13/15,14/16,15/17,16/18,17/19,18/20,19/21,20/22,21/23,22/24,23/25,24/1,25/2}
  \path (a.corner \y) edge  (a.corner \z);

\end{tikzpicture} \arrow[d] & \vdots \arrow[d] \\
% I need to change this graph
\begin{tikzpicture}

%vertices
\draw[fill=black] (0,0) circle (0.7pt);
\draw[fill=black] (0.6,0) circle (0.7pt);

%edges
\draw (0,0) edge [bend left=20] (0.6,0);
\draw (0,0) edge [bend right=20] (0.6,0);
\draw (0,0) edge [loop left, in = 155, out = 205,min distance=4mm] (0,0) ;
\draw (0.6,0) edge [loop right, in = 25, out = 335,min distance=4mm] (0.6,0) ;
\end{tikzpicture} \arrow[drrr]& \hspace{1cm}& \begin{tikzpicture}[baseline={([yshift=-0.6ex] current bounding box.center)}]
% create the node
\node[draw=none,minimum size=1cm,regular polygon,regular polygon sides=3] (a) {};

% draw a black dot in each vertex
\foreach \x in {1,2,...,3}
  \fill (a.corner \x) circle[radius=0.7pt];
  
\foreach \y\z in {1/2,2/3,3/1}
  \path (a.corner \y) edge [bend left=10] (a.corner \z);
\foreach \y\z in {1/3,2/1,3/2}
  \path (a.corner \y) edge [bend left=10] (a.corner \z);
\end{tikzpicture}  \arrow[dr]& \hspace{1cm}& \hspace{1cm}&\begin{tikzpicture}[baseline={([yshift=-0.6ex] current bounding box.center)}]
% create the node
\node[draw=none,minimum size=1cm,regular polygon,regular polygon sides=5] (a) {};

% draw a black dot in each vertex
\foreach \x in {1,2,...,5}
  \fill (a.corner \x) circle[radius=0.7pt];
  
\foreach \y\z in {1/2,2/3,3/4,4/5,5/1}
  \path (a.corner \y) edge (a.corner \z);
\foreach \y\z in {1/3,2/4,3/5,4/1,5/2}
  \path (a.corner \y) edge (a.corner \z);

\end{tikzpicture} \arrow[dll] & \begin{tikzpicture}[baseline={([yshift=-0.6ex] current bounding box.center)}]
% create the node
\node[draw=none,minimum size=1cm,regular polygon,regular polygon sides=7] (a) {};

% draw a black dot in each vertex
\foreach \x in {1,2,...,7}
  \fill (a.corner \x) circle[radius=0.7pt];
  
\foreach \y\z in {1/2,2/3,3/4,4/5,5/6,6/7,7/1}
  \path (a.corner \y) edge  (a.corner \z);
  
\foreach \y\z in {1/3,2/4,3/5,4/6,5/7,6/1,7/2}
  \path (a.corner \y) edge   (a.corner \z);

\end{tikzpicture} \arrow[dlll] \\
\hspace{1cm}& \hspace{1cm}& \hspace{1cm}& \begin{tikzpicture}[baseline={([yshift=-1.7ex] current bounding box.center)}]
% create the node
\node[draw=none,minimum size=1cm,regular polygon,regular polygon sides=1] (a) {};
% draw a black dot in each vertex
\foreach \x in {1}
  \fill (a.corner \x) circle[radius=0.7pt];
\draw (a.corner 1) to [in=50,out=130,distance = 0.5cm,loop] (a.corner 1);
\draw (a.corner 1) to [in=50,out=130,distance = 0.3cm,loop] (a.corner 1);
\end{tikzpicture}&\hspace{1cm} &\hspace{1cm} & \hspace{1cm}
\end{tikzcd}
\end{equation*}